\documentclass[11pt]{amsart}
\usepackage[utf8]{inputenc}
\usepackage{amsmath}
\usepackage{amsfonts,amssymb}
\usepackage{latexsym}

\usepackage{tikz-cd}
\usepackage{tikz}
\usetikzlibrary{decorations.pathmorphing}
\usepackage{indentfirst}

\theoremstyle{definition}

\usepackage{fancyhdr} 
\renewcommand{\sectionmark}[1]{\markright{\textit{\thesection. #1}}{} }
\newtheorem{theorem}{Theorem}[section]
\newtheorem{proposition}[theorem]{Proposition}
\newtheorem{lemma}[theorem]{Lemma}
\newtheorem{corollary}[theorem]{Corollary}
\newtheorem{definition}[theorem]{Definition}

\newtheorem{remark}[theorem]{Remark}
\newtheorem{example}[theorem]{Example}
\newtheorem{question}[theorem]{Question}
\newtheorem{open problem}[theorem]{Open problem}

\newcommand{\R}{\mathbb R}

\newcommand{\Sp}{\mathbb S}

\newcommand{\N}{\mathbb N}
\newcommand{\C}{\mathbb C}

\newcommand{\Lip}{\mathrm{Lip}}

\newcommand{\Spec}{\mathrm{Spec}}
\newcommand{\Bor}{\mathrm{Bor}}
\newcommand{\Id}{\mathrm{Id}}
\newcommand{\pr}{\mathrm{Pr}}
\newcommand{\Wass}{\mathrm{Wass}}
\newcommand{\Diam}{\mathrm{diam}}
\newcommand{\Diag}{\mathrm{Diag}}

\newcommand{\calP}{{\mathcal P}}
\newcommand{\calR}{{\mathcal R}}

\newcommand{\calA}{{\mathcal A}}
\newcommand{\calB}{{\mathcal B}}

\newcommand{\calC}{{\mathcal C}}

\newcommand{\calH}{{\mathcal H}}

\newcommand{\calT}{{\mathcal T}}

\newcommand{\arWp}[1]{\underleftarrow{W}_{#1}}
\newcounter{para}

\title{$L^p$-Wasserstein distances on state and quasi-state spaces of $C^*$-algebras}
\author{Danila Zaev}
\address{Faculty of Mathematics, Higher School of Economics, Moscow}
\date{\today}
\email{zaev.da@gmail.com}

\begin{document}

\begin{abstract}
	We construct an analogue of the classical $L^p$-Wasserstein distance for the state space of a $C^*$-algebra. Given an abstract Lipschitz gauge on a $C^*$-algebra $\calA$ in the sense of Rieffel, one can define the classical $L^p$-Wasserstein distance on the state space of each commutative $C^*$-subalgebra of $\calA$. We consider a projective limit of these metric spaces, which appears to be the space of all quasi-linear states, equipped with a distance function. We call this distance the projective $L^p$-Wasserstein distance. It is easy to show, that the state space of a $C^*$-algebra is naturally embedded in the space of its quasi-linear states, hence, the introduced distance is defined on the state space as well. We show that this distance is reasonable and well-behaved. We also formulate a sufficient condition for a Lipschitz gauge, such that the corresponding projective $L^p$-Wasserstein distance metricizes the weak$^*$-topology on the state space. 
\end{abstract}	

\maketitle
\tableofcontents

\section{Introduction}
\sectionmark{Introduction}
	We are going to connect the following three theories.
		
	\noindent \textbf{Quantum compact metric spaces}. This theory was basically developed by Marc Rieffel (see, for example, \cite{Rieffel1}, \cite{Rieffel2}, \cite{Rieffel3}). It studies state spaces of $C^*$-algebras (or, more generally, state spaces of unit-order spaces) equipped with a distance function. This distance function arise from an abstract Lipschitz gauge, defined on an algebra. Usually it is required for a distance function to metricize the weak$^*$-topology of the state space, and, hence, to provide it with a structure of a compact metric space. These ideas were pioneered by Alain Connes in the noncommutative geometry framework (see \cite{Connes}), and, as in present-day, it is one of the two most developed approaches to the noncommutative generalization of the notion of metric space (the other approach is due to Nick Weaver/Greg Kuperberg, \cite{Weaver}). Arguably, the main advantage of the Rieffel's approach is the possibility to define an analogue of a Gromov-Hausdorff distance between quantum metric spaces.
		
	\noindent \textbf{Kantorovich-Wasserstein spaces}. This theory studies spaces of probability measures equipped with a special type of distances, which are defined as solutions of a variational optimal transport problem (which is also called Kantorovich or Monge-Kantorovich problem, see \cite{AmbGigli1}, \cite{BogKol}, \cite{Villani1} for a survey). These distances depend on the distance on the base space (the space where probability measures are defined), and are parametrized by a parameter $p\in [1,\infty)$. The $L^p$-Wasserstein distance with $p=1$ is closely connected with the Connes-Rieffel distance (in the case of a commutative $C^*$-algebra, its state space can be thought of as a space of probability measures). In some sense, Connes generalized Kantorovich construction to the noncommutative case. But the $L^p$-Wasserstein distances with $p>1$ are also valuable. For example, when $p=2$, geodesics on the Wasserstein space correspond to a meaningful dynamics (see \cite{AmbGigli1}). Furthermore, the $L^p$-Wasserstein distances are used in quantitative estimations of such functionals as entropy or variation. There are several attempts to link Kantorovich theory with noncommutative geometry (\cite{CarMaas}, \cite{Martinetti1}, \cite{Martinetti2}). Some noncommutative analogues of the $L^2$- and $L^1$-Wasserstein distances were given in these papers. But the problem of constructing a meaningful analogue of the $L^p$-Wasserstein distance (for $p>1$) for the state space of noncommutative $C^*$-algebra in the general framework of Rieffel is still open.
	
	\noindent \textbf{Posets of commutative subalgebras}. One of the modern approaches to construct a mathematical foundation of quantum theory is to consider, instead of a noncommutative $C^*$-algebra of observables, its poset of commutative $C^*$-subalgebras. This approach is compatible with the so-called Bohr interpretation of quantum mechanics, and has a beautiful mathematical formalization in the language of topoi. These ideas were developed and described by different authors (see, for example, \cite{Heunen1},  \cite{Heunen2}, \cite{Isham1}), but the metric aspect of the theory has not been explored yet. When one consider commutative subalgebras instead of original noncommutative algebra, he/she inevitably loses some part of information. However, the state space of a $C^*$-algebra also contains only a part of the information about the algebra (more precisely, the part that is encoded by self-adjoint elements of the algebra, see \cite{Alfsen}). Thus, despite we potentially lose some data, it seems reasonable to define distances (in particular, analogues of the $L^p$-Wasserstein distances) on the state space of a noncommutative algebra, passing throw the poset of its commutative subalgebras. Further we show that this approach is actually justified. 
	In the topos-theoretical interpretation it is natural to consider a space of all quasi-linear states (\textquotedblleft quasi-states") instead of a state space itself (\cite{Heunen1}). Elements of this space are order- and unit-preserving maps from a $C^*$-algebra $\calA$ to the field of complex numbers $\C$ that are linear on commutative subalgebras of $\calA$. A state space naturally embeds in the corresponding quasi-state space. We do not use the language of topoi in the paper, but we provide a construction of (analogues of) the $L^p$-Wasserstein distances for both state and quasi-state spaces of $C^*$-algebras.

\textbf{Structure of the paper}. At first, we informally discuss the main ideas and formulate some problems for the further analysis. Then we provide a rigorous definition of the quasi-state space and establish its connection with the notion of the state space of an algebra. We define the quasi-state space of a $C^*$-algebra $\calA$ as the projective limit of a projective system of the state spaces of commutative $C^*$-subalgebras of $\calA$. It appears, that the state space of $\calA$ (equipped with its natural weak$^*$-topology) is topologically embedded in the corresponding quasi-state space.

Next step is to add metric information in the form of an abstract Lipschitz gauge into our framework. We recall some results of Rieffel and provide some new definitions. The reason for providing new definitions is that we are going to study distances on the state spaces of subalgebras, and to make it possible we need to modify the standard notion of a $\Lip$-seminorm. We establish some easy facts about the $L^p$-Wasserstein metrics on the state spaces of commutative $C^*$-subalgebras (which are actually spaces of probability measures), and then we define an analogue of the $L^p$-Wasserstein distance for the (quasi-)state space of the whole $C^*$-algebra $\calA$ (we call it the projective $L^p$-Wasserstein distance).

We formulate conditions on an abstract Lipschitz gauge, such that under them the corresponding projective $L^p$-Wasserstein distance metricizes the weak$^*$-topology of the state space of $\calA$. It appears, that it is sufficient for a gauge to be a $\Lip$-seminorm in the sense of Rieffel, to be finite on a dense subspace of each maximal commutative $C^*$-subalgebra of $\calA$, and to satisfy a certain inequality. Unfortunately, there is a lack of non-trivial examples of such Lipschitz gauges.

We show that the quasi-state space of a $C^*$-algebra, equipped with the projective $L^p$-Wasserstein distance, can be seen as the projective limit of the state spaces of unital commutative $C^*$-subalgebras, equipped with the ordinary $L^p$-Wasserstein distances. Besides the main results, some other facts are proved: for example, the fact, that diameters of the state space and the quasi-state space coincide under mild assumptions on a Lipschitz gauge, and the fact that the diagram of all commutative unital $C^*$-subalgebras can be recovered from the self-adjoint part of the corresponding $C^*$-algebra.

\section{Informal discussion}
\sectionmark{Informal discussion}
Let $X$ be a compact metrizable space. Then $C(X)$, the algebra of all continuous complex-valued functions, is a commutative $C^*$-algebra. It is unital, because $X$ is compact, and separable, because $X$ is metrizable. Denote by $\calP(X)$ the set of all Borel probability measures on $X$. Due to the fact, that on a metrizable compact space any Borel measure is a Radon one, $\calP(X)$ is equal to the positive part of the unit sphere of the Banach dual of $C(X)$: $\calP(X)=\Sp_1^+(C(X)^*)$. Equip $\calP(X)$ with the weak$^*$-topology. By Banach-Alaoglu theorem, it is compact, and, since $C(X)$ is separable, $\calP(X)$ is metrizable. It is also known, that $\calP(X)$ is a Choquet simplex (see \cite{Phelps}). Recall the definition of a simplex.

\begin{definition}
	\textbf{Choquet boundary} $\partial_e(K)$ of a compact convex set $K$ is a subset consisted of all extreme points, i.e. such points $a\in K$ that
	$$
	\forall a_1, a_2\in K, a_1\neq a_2:\ t\cdot a_1+(1-t)\cdot a_2=a \implies t=0\ \text{or}\ t=1
	$$
\end{definition}

\begin{definition}
	\label{choquet simplex}
	Compact convex set $K$ is a \textbf{Choquet simplex} iff for every element $a\in K$ there is a \textbf{unique} measure $\mu_a$ on $K$ such that
	\begin{itemize}
		\item $\mu_a(\partial_e(K))=1$
		\item $a=\int_K x d \mu_a$
	\end{itemize}
\end{definition}

Recall, that the \textbf{existence} of a representing measure is guaranteed for any compact convex set (see \cite{Phelps}):
\begin{theorem}[Choquet]
	For every element $a\in K$ of a compact convex metrizable set $K$, there is a probability measure $\mu_a$ on $K$ such that
	\begin{itemize}
		\item $\mu_a(\partial_e(K))=1$
		\item $a=\int_K x d \mu_a$
	\end{itemize}
\end{theorem}
\textbf{Bauer simplex} is a Choquet simplex with closed (hence compact) Choquet boundary. There are infinitely many non-isomorphic Bauer simplexes. In fact, $\calP(X)$ is an example of a Bauer simplex. Its Choquet boundary consists of all Dirac measures, and the boundary is homeomorphic to $X$ via the identification: $x \rightarrow \delta_x$. It is also known, that every Bauer simplex is isomorphic to $\calP(Y)$ for some compact Hausdorff space $Y$ (see \cite{Batty}).

Note, that Dirac measures can be characterized in the different ways:
\begin{enumerate}
	\item If $\mu\in \partial_e(\calP(X))$, then it is a Dirac measure, $\mu=\delta_x$ for some $x\in X$.
	\item If $\mu\in \calP(X)$, such that it has a unique representative measure, $\tilde{\mu}\in \calP(\partial_e(\calP(X)))$, which is concentrated on $\{\mu\}$: $\tilde{\mu}(\{\mu\})=1$, then $\mu$ is a Dirac measure.
	\item If $\mu\in \calP(X)=\Sp_1^+(C(X)^*)$ is a unital $C^*$-homomorphism between algebras $C(X)$ and $\C$: $\mu(f(x)\cdot g(x))=\mu(f(x))\cdot \mu(g(x))$, where $\mu(h(x)):=\int_X h(x) d\mu$, then it is a Dirac measure.
\end{enumerate}
The first characterization uses convex structure of $\calP(X)$. It makes sense for every convex compact set. The second characterization uses measure-theoretic structure of $\calP(X)$, and makes sense for even more general cases. In fact, for convex compact sets both the first and the second characterizations coincide (see \cite{Dynkin}). Thus, we can define generalized Dirac elements of some convex compact set as the elements of its Choquet boundary. 

The third definition uses the fact, that $\calP(X)$ is a subset of a dual space to some algebra. It uses multiplicative structure of predual, and the coincidence of this characterization with the previous ones is a property of the particular object, $\calP(X)$. One cannot expect such coincidence in any other case, but, of course, we are able to provide the following definition of generalized Dirac measures: for a subset of a Banach dual to some $C^*$-algebra, generalized \textquotedblleft Dirac measure" can be defined as an element, which is a unital $C^*$-homomorphism between the algebra and $\C$ (its continuity is automatically guaranteed). Recall, that the set of all unital $C^*$-homomorphisms from $C^*$-algebra $\calA$ to $\C$ is called the Gelfand spectrum of the algebra. It is a subset of the Banach dual of an algebra, and can be equipped with the weak$^*$-topology inherited by this inclusion: $\Spec(\calA)\subset \calA^*$. It is also true, that $\Spec(\calA)\subseteq \Sp_1^+(\calA^*)$. In the case $\calA=C(X)$, $\Spec(C(X))=\partial_e(\calP(X))$.

To sum up the above discussion, let us emphasize the following two facts:
\begin{itemize}
	\item Any compact metrizable space $X$ is homeomorphic to $\partial_e(\calP(X))$ and homeomorphic to $\Spec(C(X))$, where $\calP(X)$ is the set of all Borel probability measures on $X$. In the case of non-metrizable compact Hausdorff space $X$, we should define $\calP(X)$ as the set of all Radon probability measures. 
	\item For any commutative unital $C^*$-algebra $\calA$ it is true, that $\calP(X)=\Sp_1^+(\calA^*)$, and $\partial_e(\calP(X))$ appears to be homeomorphic to $\Spec(\calA)=\Spec(C(\Spec(\calA)))$. Moreover, $\Spec(\calA)$ is compact and, if  $\calA$ is separable, it is metrizable.
\end{itemize}

We are going to add some metric information in our framework. It can be done in different ways. One way is to define distance function on $X$, which metricizes its topology. The other way is to define a Lipschitz gauge on $C(X)$. Following Rieffel, we define Lipschitz gauge as a partially-defined seminorm on the self-adjoint part of an algebra, which satisfies some natural requirements. In the following we provide several rigorous definitions of Lipschitz gauges (or, more precisely, collections of requirements a seminorm should satisfy to be a Lipschitz gauge). As for now, let us consider that $L: \calA^{sa} \rightarrow [0,+\infty]$ is an abstract Lipschitz gauge on a separable unital $C^*$-algebra $\calA$ ($\calA^{sa}$ is the subspace of all self-adjoint elements of $\calA$), if it is
\begin{enumerate}
	\item absolutely homogeneous: $L(af)=|a|L(f)$, for any $a\in \R$, $f\in \calA^{sa}$,
	\item subadditive: $L(f+g)\leq L(f)+L(g)$, for any $f, g\in \calA^{sa}$,
	\item lower semi-continuous: $\{f\in \calA^{sa}: L(f)\leq t\}$ is norm closed for one, hence every, $t\in (0,+\infty)$,
	\item $L(a)=0$ iff $a\in 1\cdot \R$,
	\item $d_L(\mu,\nu):=\sup\{|\mu(f)-\nu(f)|: f\in \calA^{sa},\ L(f)\leq 1\}$ is a distance function on $\Sp_1^+(\calA^*)$ (positive part of a unit sphere of the Banach dual of $\calA$), which metricizes its weak$^*$-topology.
\end{enumerate}
Given an abstract Lipschitz gauge $L$ on $\calA=C(X)$, we are able to define a distance $d_L$ on $\calP(X)=\Sp_1^+(C(X)^*)$, and, restricting it to $\partial_e(\calP(X))=X$, we define a distance function $d$ on $X$: $d(x,y)=d_L(\delta_x,\delta_y)$. 
Hence, it is possible to recover a distance on $X$ from an abstract Lipschitz gauge on $\calA=C(X)$. 
We can define a new Lipschitz gauge by the formula
$$
L_d^e(f):=\sup\left\{\frac{f(x)-f(y)}{d(x,y)}: x\neq y,\ x,y\in X\right\}
$$
The distance function induced by $L_d^e$ coincides with the \textbf{$L^1$-Wasserstein distance}, which is defined \textquotedblleft in the dual way":
\begin{multline}
d_e(\mu,\nu):=\sup\{|\mu(f)-\nu(f)|: f\in \calA^{sa},\ L_d^e(f)\leq 1\}=\\
=W_1(\mu,\nu):=\inf\left\{\int d(x,y) d\pi: \pi\in \calP(X\times Y), \mathrm{Pr}(\pi)=(\mu,\nu)\right\}
\end{multline}
where $\pr: \calP(X\times Y)\rightarrow \calP(X)\times \calP(Y)$ is defined as $\pr (\pi):=((\pr_X)_\#(\pi),(\pr_Y)_\#(\pi))$. This equality is known as \textquotedblleft Kantorovich duality" (or Kantorovich-Rubinstein formula), and it is a well-known result in the mass transportation theory (see \cite{AmbGigli1}, \cite{BogKol}, \cite{Villani1}). It allows one to define a distance on $\calP(X)$ via minimization over the set of transport plans (probability measures on $X\times X$ with fixed marginals).
Analogously, we can define the $L^p$-Wasserstein distances on $\calP(X)$ for $p\in [1,+\infty)$ using the definition via transport plans:
$$
W_p(\mu,\nu):=\inf\left\{\left(\int d^p(x,y) d\pi\right)^\frac{1}{p}: \pi\in \calP(X\times Y), \mathrm{Pr}(\pi)=(\mu,\nu)\right\}
$$
It appears, that all these distances metricize the weak$^*$-topology on $\calP(X)$ (see \cite{BogKol}). Some of the $W_p$-distances have additional good properties (see \cite{BogKol} for details).

It was also proved by Rieffel (see Theorem 8.1 in \cite{Rieffel2}), that under some additional assumptions about Lipschitz gauge, $L=L_d^e$ on $C(X)$. This result is important for us, but we shall discuss the details later.

To sum up, a pair of a commutative separable $C^*$-algebra $\calA$ and an abstract Lipschitz gauge $L$ on $\calA$ determines a Wasserstein metric space $(\calP(X), W_p)$ for each $p\in [1,+\infty)$. In the following sections we shall also consider non-separable $C^*$-algebras, which have non-metrizable state spaces. In this case the pairs $(\calP(X), W_p)$ should be thought as pairs of topological spaces and lower semi-continuous distances on them.

Consider a more general case of possibly non-commutative unital separable $C^*$-algebra $\calA$. In this case we define the state space as the set $S(\calA):=\Sp_1^+(\calA^*)$. It is a direct analogue of a set of probability measures. It can be equipped with the corresponding weak$^*$-topology, and appears to be a compact convex set. However, it is known (see \cite{Batty}), that $S(\calA)$ is a simplex if and only if $\calA$ is commutative. Since every space of probability measures $\calP(X)$ is a simplex, in the noncommutative case there is no analogue of a space $X$: $S(\calA)\neq \calP(X)$ for any $X$, if $\calA$ is not commutative. 

An element of the Choquet boundary of $S(\calA)$, $\partial_e(S(\calA))$, is called a pure state on $\calA$. As follows from the Choquet theorem, every state $\mu$ can be represented by $\tilde{\mu}\in \calP(\partial_e(S(\calA)))$, $\mathrm{bar}(\tilde{\mu})=\mu$ (\textquotedblleft bar" = barycenter), but this representation is not unique in general. 

As we noted earlier, in the commutative case, $X$ can be defined as the boundary of $S(\calA)$ or, equivalently, as the Gelfand spectrum $\Spec(\calA)$. In the noncommutative case, $\partial_e(S(\calA))$ does not coincide with $\Spec(\calA)$. Moreover, for the algebra $\calB(\calH)$ of all bounded operators on a Hilbert space, $\Spec(\calB(\calH))$ is empty. It is possible to define \textquotedblleft spectrum" of a noncommutative $C^*$-algebra in several different ways (e.g. as a primitive ideal space, as a space of equivalence classes of irreducible representations, etc.), but since $S(\calA)\neq \calP(X)$ for any $X$, there is no much sense to do that, at least, if our goal is to construct an analogue of the Wasserstein theory for the state space.

Consider a commutative unital $C^*$-subalgebra $\calA_\alpha\subseteq \calA$. Using the results described above, we can define its state space $S(\calA_\alpha)=\calP(X_\alpha)$ and its Gelfand spectrum $X:=\Spec(\calA_\alpha)$, which is homeomorphic to $\partial_e(\calP(X_\alpha))$. If $\calA_\alpha$, $\calA_\beta$ are two commutative unital $C^*$-subalgebras of $\calA$ and $\calA_\alpha\subseteq \calA_\beta$, then there are surjective continuous maps:
\begin{eqnarray}
X_\beta& \twoheadrightarrow&  X_\alpha\\
\calP(X_\beta)& \twoheadrightarrow&  \calP(X_\alpha)
\end{eqnarray}
The first map is actually the restriction map from $\calA_\beta$ to $\calA_\alpha$ of elements of the Gelfand spectrum. The second map is the restriction map of continuous linear functionals. These maps are onto due to the Gelfand duality (each injective $C^*$-homomorphism of commutative $C^*$-algebras corresponds to a surjective continuous map between their Gelfand spectrums). In the case of state spaces, the map appears to be affine:
$$
t\cdot \mu_\beta + (1-t)\cdot \nu_\beta \rightarrow t\cdot \mu_\alpha + (1-t)\cdot \nu_\alpha,\ t\in [0,1].
$$

We are able to consider an ordered set of all unital commutative $C^*$-subalgebras $\calA_\alpha$ of a $C^*$-algebra $\calA$, where ordering is defined by the relation of inclusion. It can be thought of as a diagram in the category of commutative unital $C^*$-algebras, where morphisms are injective unital $C^*$-homomorphisms. We can apply Gelfand duality to obtain a diagram in the dual category of Hausdorff compact spaces with surjective continuous maps, and then apply a functor that maps each compact space to the set of all probability measures on it (equipped with the weak$^*$-topology) and sends each map $f: X_\beta \rightarrow X_\alpha$ to the map $f_\#: \calP(X_\beta) \rightarrow \calP(X_\alpha)$, which is a pushforward of measures. We obtain a diagram in the category of convex compact spaces and affine continuous maps. It is natural to consider the projective limit of this diagram. It exists and appears to be isomorphic to the space $QS(\calA)$ of all quasi-states (quasi-linear states) of the algebra $\calA$. Here $QS(\calA)$ is supposed to be equipped with a projective topology, inherited from the inclusion $QS(\calA)\subseteq \prod_\alpha S(\calA_\alpha)$.
\begin{definition}
	\label{quasi-state}
	A \textbf{quasi-state} on a $C^*$-algebra $\calA$ is a map $\mu: \calA \rightarrow \C$, such that it is linear on all unital commutative $C^*$-subalgebras of $\calA$, satisfies $\mu(a+ib)=\mu(a)+i\mu(b)$ for all self-adjoint $a, b\in \calA^{sa}$, $\mu(a^*a)\geq 0$ for all $a\in \calA$, and $\mu(1)=1$.
\end{definition}
It is clear, that the space of all quasi-states is convex. Since all $S(\calA_\alpha)$ are compact Hausdorff, $\prod_\alpha S(\calA_\alpha)$ is compact and Hausdorff too, and $QS(\calA)\subseteq \prod_\alpha S(\calA_\alpha)$ is closed,
hence compact and Hausdorff. Each quasi-state can be restricted to a subalgebra $\calA_\alpha$, and this restriction defines a continuous affine surjection $QS(\calA)\twoheadrightarrow S(\calA_\alpha)$.

It is clear, that the state space $S(\calA)$ is a subset of the corresponding quasi-state space: $S(\calA) \subseteq QS(\calA)$. It is natural to ask, does the weak$^*$-topology on $S(\calA)$ coincide with the topology obtained by this inclusion?
\begin{question}
	Does the image of the inclusion $S(\calA) \subseteq QS(\calA)$ homeomorphic to $S(\calA)$?
\end{question}
The answer is affirmative, as it will be shown in the next section (Corollary \ref{topologies on S and QS coincide}).

There are cases, when $S(\calA)=QS(\calA)$. It is stated by the Gleason's theorem, that,
if $\calA$ is a von Neumann algebra without a direct factor isomorphic to $\calB(\C^2)$, every quasi-state of $\calA$ is a state. In particular, it is true for every $\calA=\calB(\calH)$, if $\dim \calH>2$ (see \cite{Bunce} for details).

Let us add some metric information to our $C^*$-algebra $\calA$. As in the commutative case, we can consider abstract Lipschitz gauge defined on the algebra. In fact, we can use exactly the same definition of a gauge, as the one given above (as one can note, it does not require commutativity of a $C^*$-algebra). As follows from the definition, we are able to define the distance:
$$
d_L(\mu,\nu):=\sup\{|\mu(f)-\nu(f)|: f\in \calA^{sa},\ L(f)\leq 1\}
$$
on $S(\calA)$, which metricizes its weak$^*$-topology. In the noncommutative geometry it is known by the names \textquotedblleft Connes' distance" or \textquotedblleft spectral distance". The last name is motivated by the fact, that usually $L$ is defined as $L(f):=||[D,f]||$ for some Dirac operator $D$ (see \cite{Connes} for definitions and details). This distance can be thought of as a generalization of the Kantorovich distance. The problem is that there is no clear way to construct an analogue of the $L^p$-Wasserstein distance for $S(\calA)$, when $\calA$ is a (possibly noncommutative) unital $C^*$-algebra.

Let us define a distance $d_{L, \alpha}$ on $S(\calA_\alpha)$ for a commutative unital subalgebra $\calA_\alpha\subseteq \calA$ in the following way:
$$
d_{L, \alpha}(\mu_\alpha,\nu_\alpha):=\sup\{|\mu(f)-\nu(f)|: f\in \calA^{sa}_\alpha,\ L(f)\leq 1\}
$$
Due to the fact, that $S(\calA_\alpha)=\calP(\Spec(\calA_\alpha))$, we can define the $L^p$-Wasserstein distances on $S(\calA_\alpha)$ for $p\in [1,+\infty)$ by the formula:
$$
W_{p,\alpha}(\mu,\nu):=\inf\left\{\left(\int d_{L,\alpha}^p(x,y) d\pi\right)^\frac{1}{p}: \pi\in \calP(\Spec(\calA_\alpha)\times \Spec(\calA_\alpha)),\ \mathrm{Pr}(\pi)=(\mu,\nu)\right\}
$$

The natural questions are:
\begin{enumerate}
	\item Is $W_{p, \alpha}$ a distance on $S(\calA_\alpha)$ for any $\calA_\alpha$?
	\item Is it true, that $W_{p,\alpha}\leq W_{p,\beta}$ iff $\calA_\alpha\subseteq \calA_\beta$?
	\item Does $W_{1,\alpha}=d_{L,\alpha}$?
	\item Does $W_{p, \alpha}$ metricize the weak$^*$-topology on $S(\calA_\alpha)$?
\end{enumerate}
These questions will be discussed later in the paper, and the answers for them can be found on one of the next sections.

Let us define a distance on the quasi-state space $QS(\calA)$ (and, hence, on the state space $S(\calA)$ as well) as follows:
$$
\arWp{p}(\mu,\nu):=\sup\{W_{p,\alpha}(\mu_\alpha,\nu_\alpha): \calA_\alpha\subseteq \calA\},
$$
where $\calA_\alpha$ is a commutative unital $C^*$-subalgebra of $\calA$, $\mu_\alpha:=\mu|_{\calA_\alpha}$ is the restriction of a quasi-state $\mu$ from $\calA$ to $\calA_\alpha$.

Again, we have a list of natural questions about this just defined object:
\begin{enumerate}
	\item Is $\arWp{p}$ a distance function on $QS(\calA)$ (or on $S(\calA)$)?
	\item Does $\arWp{p}<\infty$ on $QS(\calA)$ ($S(\calA)$)? 
	\item Does $\arWp{1}=d_L$ on $S(\calA)$?
	\item Does $\arWp{p}$ metricize the weak$^*$-topology on $S(\calA)$?
\end{enumerate}
It appears, that under some assumptions about an abstract Lipschitz gauge, the answer to these questions is affirmative. We call $\arWp{p}$ the projective $L^p$-Wasserstein distance.

We shall also provide a positive answer to the following question:
\begin{question}
	Is it possible to define a category, such that $(QS(\calA), \arWp{p})$ appears to be a projective limit of the diagram $\{(S(\calA_\alpha), W_{p,\alpha})\}$ (where the order is defined in a natural way)?
\end{question}

\section {Space of quasi-states}
\sectionmark{Space of quasi-states}

In this section we provide a rigorous definition of the quasi-state space of a unital $C^*$-algebra $\calA$. It appears to be a projective limit of a projective system of the state spaces of unital commutative $C^*$-subalgebras of $\calA$.

Define the category of all unital commutative $C^*$-algebras with injective unital $C^*$-homomorphisms between them. Denote it by $\bf ucC^*_{in}$. Note, that due to the $C^*$-structure, these homomorphisms are isometric. It follows from the Gelfand duality, that this category is anti-equivalent to the category of all compact Hausdorff spaces with continuous surjections ($\bf CH_{sur}$):

	\begin{tikzcd}
		\bf {ucC^*_{in}}^{op}
		\arrow[bend left]{r}{Spec}
		& \bf CH_{sur}
		\arrow [bend left]{l}{C(\cdot)}
	\end{tikzcd}

Here $\Spec$ is a functor that maps each algebra into its Gelfand spectrum. The (weak) inverse functor to $\Spec$ is a functor $C(\cdot)$, which sends a compact space to its corresponding commutative $C^*$-algebra of all $\C$-valued continuous functions: $\calA\simeq C(\Spec(\calA))$.

Let $\bf BS_{sur}$ be a category of all Hausdorff Bauer simplexes (Choquet simplexes with closed Choquet boundary, see Definition \ref{choquet simplex}) with surjective continuous affine maps. Consider a functor $\calP: \bf CH_{sur} \rightarrow \bf BS_{sur}$, which maps each compact Hausdorff space to the space of all Radon probability measures on it, equipped with the weak$^*$-topology, and each continuous surjection $T: X\twoheadrightarrow Y$ sends to a continuous affine surjection $T_\#: \calP(X)\twoheadrightarrow \calP(Y)$, which is defined as follows: $T_\#(\mu)(B):=\mu(T^{-1}(B))$, $\forall B\in \Bor(Y)$, $T^{-1}(B)$ is the preimage of $B$, $\mu\in \calP(X)$. The (weak) inverse of this functor is $\partial_e$, which associates to every Bauer simplex its Choquet boundary. It is a standard fact, that a space of all Radon probability measures on a compact Hausdorff space is a Hausdorff Bauer simplex (see \cite{Alfsen}, \cite{Bunce}, \cite{Phelps} for details). The fact, that $\partial_e\circ \calP \simeq \Id_{\bf CH_{sur}}$ follows from the fact, that the Choquet boundary coincides with the set of all Dirac measures, which are exactly the elements of the Gelfand spectrum. An arguably non-trivial fact here is the equivalence $\calP\circ \partial_e \simeq \Id_{\bf BS_{sur}}$ on the level of morphisms. Let us check it.

\begin{proposition}
	Let $T: \calP(X)\twoheadrightarrow \calP(Y)$ be a morphism in $\bf BS_{sur}$, $T|_X$ be the restriction of $T$ to $X$, $(T|_X)|_\#: \calP(X)\rightarrow \calP(Y)$ be the corresponding pushforward of measures. Then $T|_X$ is a morphism between $X$ and $Y$ in $\bf CH_{sur}$, and $(T|_X)_\#=T$.
\end{proposition}
\begin{proof}
	$T|_X$ is, by definition, a continuous map from $X$ to $\calP(Y)$. Since $T$ is affine, it sends extreme points to extreme points, hence, $T|_X: X\rightarrow Y$. By surjectivity, contunuity, and affinity of $T$, $T^{-1}(y)$ is a nonempty closed face of $\calP(X)$. It is standard, that extreme points of a face are extreme in the whole convex set. Hence, $T^{-1}(y)\cap X$ is not empty, and $T|_X$ is surjective.
	
	By Gelfand duality, $(T|_X)^*: C(Y)\rightarrow C(X)$ is a morphism in $\bf ucC^*_{in}$. For every $f\in C(Y)$, $(T|_X)^*(f):=f\circ T|_X \in C(X)$ and $\int (f\circ T|_X) d\mu = \int f dT(\mu)$ for every $\mu\in \calP(X)$. By the Rietz representation theorem, there is a bijective correspondence between $\mu\in \calP(Y)$ and positive continuous linear functional on $\calC(Y)$. Since
	$$
	((T|_X)_\#\mu)(f):=\int_Y f(y) d((T|_X)_\#\mu)=\int_X f(T|_X(x))d\mu=\int_Y f(y)dT(\mu)=T(\mu)(f)
	$$
	is true for every $f\in C(Y)$, the statement of the proposition follows.
\end{proof}

Let us call a topology on a convex set $K$ \textbf{compatible} with the convex structure iff the operation: $A_t: K\times K \rightarrow K$, $A_t(a,b)=t\cdot a+ (1-t) \cdot b$ is continuous for every $t\in[0,1]$. By \textbf{compact convex space} we shall mean a convex set equipped with a compatible topology, such that it appears to be compact. Define a category of convex compact Hausdorff spaces with affine continuous maps between them. We denote it as $\bf CCH$. It is clear, that $\bf BS_{sur}$ is a subcategory of $\bf CCH$. Hence, there exists an injective on objects faithful functor from $\bf BS_{sur}$ to $\bf CCH$, associated with the inclusion. Denote it by $i$.

The following diagram indicates the relationship between categories. Functor $S: \bf {ucC^*_{in}}^{op} \rightarrow \bf BS_{sur}$ associates with an algebra the unit sphere of its Banach dual space, and sends every morphism of algebras to the respective dual (adjoint) map.

\begin{tikzcd}
	\bf {ucC^*_{in}}^{op}
	\arrow[bend left=50]{r}{\Spec}
	\arrow[bend right=100]{rr}{S}
	& \bf CH_{sur}
	\arrow[bend left=50, swap]{l}{C(\cdot)}
	\arrow[bend left=50]{r}{\calP}
	& \bf BS_{sur}
	\arrow[bend left=50, swap]{l}{\partial_e}
	\arrow[hookrightarrow]{r}{i}
	& \bf CCH
\end{tikzcd}

Let us state some facts about the category $\bf CCH$.
\begin{proposition}
	For any set of objects $\{K_\alpha\}$ in $\bf CCH$, there is a categorical product $\prod_\alpha K_\alpha$, which is defined as a topological product with the natural convex structure:
	$$
	t\cdot a + (1-t) \cdot b = (t\cdot a_\alpha + (1-t) \cdot b_\alpha)
	$$
	where $a=(a_\alpha)$, $b=(b_\alpha)$, $t\in [0,1]$. Projection maps are projections in the usual sense.
\end{proposition}
\begin{proof}
	A product convex structure is compatible if each factor has a compatible structure. It follows from the fact, that the product topology is the topology of pointwise convergence, and $(a,b)\rightarrow t\cdot a + (1-t) \cdot b$, $t\in [0,1]$, is continuous iff every $(a_\alpha, b_\alpha) \rightarrow t\cdot a_\alpha + (1-t) \cdot b_\alpha$ is, which is exactly our case.
	It is a standard fact, that the defined product is a compact Hausdorff space. Moreover, projection maps are continuous and affine.
	
	For every such $K\in \bf CCH$, that there exists a morphism $\varphi_\alpha: K \rightarrow K_\alpha$ in $\bf CCH$ for each element $K_\alpha\in \{K_\alpha\}$, define $\varphi: K \rightarrow \prod_\alpha K_\alpha$ as follows:
	$$
	\varphi(x):=(\varphi_\alpha(x)),\ \forall x\in K
	$$
	It is straightforward to check, that this map is a continuous affine one, and that $\varphi_\alpha=Pr_\alpha\circ\varphi$. Suppose, that there is a morphism $\psi:K\rightarrow \prod_\alpha K_\alpha$, such that $\pr_\alpha \circ \psi=\varphi_\alpha$ for all $\alpha$. Then for any $x\in K$, $\pr_\alpha(\psi(x))=\varphi_\alpha(x)$, so that $(\varphi(x))_\alpha=\varphi_\alpha(x)=(\psi(x))_\alpha$. Hence, $\varphi=\psi$, and $\varphi$ is a unique morphism with this property.
\end{proof}

\begin{proposition}
	\label{existence of inverse limits in CCH}
	For any small diagram $(\{K_\alpha\}, \calT)$, $\calT=\{T_{\alpha,\beta}: K_\alpha \rightarrow K_\beta,\ \alpha\leq \beta\}$ in $\bf CCH$  (a diagram = a functor from a poset to $\bf CCH$) there exists a projective limit. Hence, the category $\bf CCH$ is complete.
	
	Projective limit is defined as follows:
	$$
	\varprojlim (\{K_\alpha\}, \calT) = \left\{x\in \prod_\alpha K_\alpha: \forall T_{\alpha,\beta}\in \calT,\ x_\beta=T_{\alpha,\beta}(x_\alpha) \right\}
	$$
	where $x=(x_\alpha)$, and equipped with the induced topological and convex structure.
\end{proposition}
\begin{proof}
	A convex structure on a projective limit is well-defined due to the affinity of morphisms $T_{\alpha,\beta}$. 
	
	Let us prove that $\varprojlim \{K_\alpha\}$ is a closed subspace of $\prod_\alpha K_\alpha$. Let $x = (x_\alpha)$ be in $\prod_\alpha K_\alpha$, but not in $\varprojlim \{K_\alpha\}$, i.e. we can find $(\alpha, \beta): K_\alpha \rightarrow K_\beta$, s.t. $x_\beta\neq T_{\alpha,\beta}(x_\alpha)$. Since $K_\beta$ is Hausdorff, we can find disjoint open neighbourhoods $V_\alpha$ of $x_\alpha$ and $U_\beta$ of $T_{\alpha,\beta}(x_\alpha)$. Since $T_{\alpha,\beta}$ is continuous, $U_\alpha:=T_{\alpha,\beta}^{-1}(V_\beta)$ is an open neighbourhood of $x_\alpha$. Thus $U_\beta\times U_\alpha\times \prod_{\gamma\not\in \{\alpha,\beta\}} K_\gamma$
	is an open neighbourhood of $x$ in $\prod_\alpha K_\alpha$ that does not intersect with $\varprojlim (\{K_\alpha\}, \calT)$.

	Suppose we have such an object $K$ and a family of morphisms $\varphi_\alpha: K\rightarrow K_\alpha$, that $\varphi_\beta=T_{\alpha,\beta}\circ \varphi_\alpha$ iff $\alpha\leq \beta$. Define $\varphi: K\rightarrow \varprojlim K_\alpha$ by $\varphi(x)=(\varphi_\alpha(x))$. Then $(\pr_\alpha\circ \varphi)(x)=\pr_\alpha((\varphi_\alpha(x)))=\varphi_\alpha(x)$, hence $\pr_\alpha\circ \varphi=\varphi_\alpha$. Suppose that there is a morphism $\psi:K\rightarrow \varprojlim K_\alpha$ such that $\pr_\alpha \circ \psi=\varphi_\alpha$ for all $\alpha$. Then for any $x\in K$, $\pr_\alpha(\psi(x))=\varphi_\alpha(x)$, so that $(\varphi(x))_\alpha=\varphi_\alpha(x)=(\psi(x))_\alpha$. Hence $\varphi=\psi$, and $\varphi$ is a unique morphism with this property.
\end{proof}

Recall, that a unital $C^*$-subalgebra of a $C^*$-algebra $\calA$ is a subset of $\calA$ that includes identity element and appears to be a unital $C^*$-algebra with respect to the inherited multiplication, involution, and linear structure. Since the inclusion of such a subalgebra in $\calA$ is an injective unital $C^*$-homomorphism, which is inevitably isometric, the Banach structure of a subalgebra coincides with the inherited one.

For a $C^*$-algebra $\calA$ we can consider a diagram $\calC(\calA)$ in $\bf ucC^*_{in}$ of all unital commutative $C^*$-subalgebras of $\calA$ ordered by inclusion. Denote it by $\calC(\calA):=(\{\calA_\alpha\}, \subseteq)$. It is clear, that every inclusion of subalgebras is an injective unital $C^*$-homomorphism of $C^*$-algebras. 

Using the defined above functors $\Spec$ and $i\circ S$ we obtain a diagram in the category $\bf CH_{sur}$ and a diagram in $\bf CCH$. We denote them as $\Spec(\calC(\calA)):=(\{\Spec(\calA_\alpha)\}, \calR)$ and $S(\calC(\calA)):=(\{S(\calA_\alpha)\}, \calR_\#)$ respectively. For shortness, we shall denote $\Spec(\calA_\alpha)$ as $X_\alpha$, thus, in this notation, $\Spec(\calC(\calA)):=(\{X_\alpha\}, \calR)$.
The ordering is defined by the restriction maps: $X_\alpha\preceq X_\beta$ iff $\exists R_{\beta, \alpha}: X_\beta \twoheadrightarrow X_\alpha$, $R_{\beta, \alpha}(x_\beta):=x_\beta|_{\calA_\alpha}$; $S(\calA_\alpha)\preceq S(\calA_\beta)$ iff $\exists (R_{\beta, \alpha})_\#: S(\calA_\alpha) \twoheadrightarrow S(\calA_\beta)$, $(R_{\beta, \alpha})_\#(\mu_\beta):=\mu_\beta|_{\calA_\alpha}$. Note, that the restriction maps commute with the functor $\calP$: $(R_{\beta,\alpha})_\#\circ \calP=\calP \circ R_{\beta,\alpha}$.

Let us define the \textbf{quasi-state space} $QS(\calA)$ for a $C^*$-algebra $\calA$ as a projective limit of a diagram $(\{S(\calA_\alpha)\}, \calR_\#)$ in $\bf CCH$. A little bit later we prove, that elements of the quasi-state space are quasi-states in the sense of Definition \ref{quasi-state}.

Recall, that an element $a$ of a $C^*$-algebra is \textbf{self-adjoint} iff $a=a^*$. Denote the ordered Banach space of all self-adjoint elements of $\calA$ by $\calA^{sa}$ (Banach and order structures are inherited from $\calA$). It is known (see \cite{Alfsen}) that $\calA^{sa}$ is a Jordan Banach algebra (JB-algebra, for a definition see \cite{Alfsen}) with respect to the product, defined as an anticommutator $a\circ b:=\frac{1}{2}(a\cdot b+b\cdot a)$. In particular, it is a commutative non-associative unital Banach algebra over $\R$. This Jordan structure is defined in terms of the multiplicative structure of $\calA$, but it can be also recovered from the order structure of $\calA^{sa}$. The order on $\calA^{sa}$, defined by the Jordan product ($a \succeq b \iff \exists c\in \calA^{sa} \text{ s.t. } a-b=c\circ c$) coincides with the order, inherited from $\calA$.

Let us summarize some easy facts about self-adjoint part of an algebra in the following Lemma.
\begin{lemma}
	\label{lemma about self-adjoint part}
	Let $\calA$ be a unital $C^*$-algebra, $\calC(\calA)$ be a diagram (in $\bf ucC^*_{in}$) of all unital commutative $C^*$-subalgebras of $\calA$ (with inclusion relation). Then
	\begin{enumerate}
		\item For every $\calA_\alpha\in \calC(\calA)$ its self-adjoint part, $\calA_\alpha^{sa}$, is a Banach unital associative subalgebra of the JB-algebra $\calA^{sa}$.
		\item $\calA^{sa}$ is covered by $\{\calA_\alpha^{sa}$: $\calA_\alpha\in \calC(\calA)\}$.
		\item $S(\calA)$ is determined by $\calA^{sa}$.
		\item $\calC(\calA)$ is determined by $\calA^{sa}$.
	\end{enumerate}
\end{lemma}
\begin{proof}
	Let us prove these assertions.
	\begin{enumerate}
		\item By Gelfand duality, $\calA_\alpha\simeq C(\Spec(\calA_\alpha))$, where $\simeq$ is an isomorphism of $C^*$-algebras. The self-adjoint part in this identification corresponds to the set of all continuous real-valued functions on $\Spec(\calA_\alpha)$. It obviously has a structure of a unital commutative associative algebra over $\R$, and, being equipped with an inherited norm, it becomes a Banach algebra: the condition $||a\cdot b||\leq ||a||\cdot||b||$ holds since it holds for $\calA$, completeness follows from the fact, that the space of all continuous real-valued functions on a Hausdorff compact space is complete w.r.t. uniform topology, which is exactly the topology determined by the norm $||\cdot||$. $\calA_\alpha^{sa}$ is a Banach subalgebra of $\calA^{sa}$, since it is a Banach subalgebra of $\calA$, and $x\circ y=x\cdot y$ for commutative elements of $\calA^{sa}$.
		\item Every $a\in \calA^{sa}$ generates a unital commutative $C^*$-subalgebra of $\calA$ (it can be defined as a completion of the set of all polynomials in $a$ over $\C$). Denote it by $\calA_\alpha$. It is clear, that $a\in \calA_\alpha^{sa}$.		
		\item Every element $\mu$ of $S(\calA)$ is defined by its values on $\calA$. Note, that any $f\in \calA$ can be represented as a linear combination of two self-adjoint elements: $a=\frac{a+a^*}{2}-i\frac{i(a-a^*)}{2}$. Then $\mu(a)=\mu(\frac{a+a^*}{2})-i\mu(i\frac{a-a^*}{2})$. Hence $\mu$ is determined by its values on $\calA^{sa}$. 
		
		The weak$^*$-topology on $S(\calA)$ is defined as the weakest one, such that for every $f\in \calA$, $\mu \rightarrow \mu(f)$ is a continuous functional on $S(\calA)$. Since $\mu(f)=\mu(a)+i\mu(b)$ for some $a,b\in \calA^{sa}$, the weak$^*$-topology can be equivalently defined as the weakest topology, such that for every $a\in \calA^{sa}$, $\mu \rightarrow \mu(a)$ is a continuous functional.
		
		The inverse statement ($\calA^{sa}$ can be recovered from $S(\calA)$ as an ordered Banach space or, equivalently, as a JB-algebra) is also true, but we do not prove it here (in case of interest, see \cite{Alfsen}).		
		\item 
		Consider associative unital Banach subalgebras of $\calA^{sa}$. By Theorem 1.12 of \cite{Alfsen}, each of them is isometrically isomorphic to the algebra (and ordered Banach space) of all real-valued continuous functions (equipped with the uniform norm) on some Hausdorff compact space $X$.
		Let $F\subseteq \calA^{sa}$ be such a subalgebra. We wish to ensure, that the minimal $C^*$-subalgebra of $\calA$ containing $F$ (let us denote it as $C(F)$) is a unital commutative $C^*$-subalgebra of $\calA$ with the self-adjoint part $F$. As a $\C$-vector space, $C(F)$ contains the space $\{f_1+i f_2: f_1, f_2\in F\}\subseteq \calA$. By axiomatic definition of a $C^*$-algebra, the involution, the multiplication, and the $C^*$-norm are uniquely defined on this set, making it a unital commutative $C^*$-algebra:
		\begin{enumerate}
			\item $(f_1+i f_2)^*=f_1-i f_2$ due to self-adjointness of $f_1$, $f_2$.
			\item $(f_1+i f_2) (f_3+ i f_4)=f_1 f_3 + i f_2 f_3 + i f_1 f_4 - f_2 f_4$ due to the distributive law. This multiplication is commutative due to commutativity of $f_k$, $k=1,...4$.
			\item $||(f_1+i f_2)(f_1+i f_2)^*||=||f_1^2+f_2^2||=||f_1+i f_2||^2$ by definition of $C^*$-norm. Due to positivity of $f_1^2+f_2^2$, $||f_1+i f_2||=||(f_1^2+f_2^2)^{\frac{1}{2}}||$.
			\item $\{f_1+i f_2: f_1, f_2\in F\}$ is complete w.r.t. this norm. Let $(f_k+i g_k)$ be a Cauchy sequence. Since $||((f_k - f_n)^2+(g_k-g_n)^2)^{\frac{1}{2}}||=||((f_k - f_n)^2+(g_k-g_n)^2)||^{\frac{1}{2}}$, $||(f_k - f_n)^2+(g_k-g_n)^2||<\varepsilon^2$ implies $||(f_k - f_n)^2||<\varepsilon^2$, $||(g_k - g_n)^2||<\varepsilon^2$. Using the identity $||a||^2=||a^2||$, which is satisfied in a JB-algebra, $||f_k - f_n||<\varepsilon$, $||g_k - g_n||<\varepsilon$, hence $(f_k)$, $(g_k)$ are Cauchy sequences in $F$. Let $f, g \in F$ be their respective limits. Then it follows from the inequality $||(f-f_k)^2+(g-g_k)^2||\leq ||(f-f_k)^2|| + ||(g-g_k)^2||= ||(f-f_k)||^2 + ||(g-g_k)||^2$, that $f+i g$ is a limit for $(f_k+i g_k)$.
		\end{enumerate}
		Due to the uniqueness, the defined commutative $C^*$-structure on $\{f_1+i f_2: f_1, f_2\in F\}\subseteq \calA$ coincides with the inherited one from $\calA$. Moreover, it coincides with the $C^*$-structure of the standard Banach complexification of $C_\R(X)\simeq F$, which is isometrically isomorphic to $C_\C(X)$. It is clear in this representation, that $F$ is the set of its self-adjoint elements. Combining this result with the first statement of this Lemma, we conclude, that every $\calA_{\alpha}\in \calC(\calA)$ can be constructed this way. The inclusions of $C^*$-subalgebras correspond to the inclusions of their self-adjoint parts.
	\end{enumerate}
\end{proof}

\begin{proposition}
	The quasi-state space is a set of quasi-states in the sense of Definition \ref{quasi-state} equipped with the natural (element-wise) convex structure and the projective topology: the weakest (coarsest) topology, such that for every unital commutative $C^*$-subalgebra $\calA_\alpha\subseteq \calA$, all linear functionals of the form $\mu_\alpha\rightarrow \mu_\alpha(f)$ for all $f\in \calA_\alpha$ are continuous. Here $\mu_\alpha$ is a restriction of $\mu\in QS(\calA)$ to $\calA_\alpha$.
\end{proposition}
\begin{proof}
	By the definition of an element of $QS(\calA)$, its restriction to $\calA_\alpha$ should be a state on each $\calA_\alpha$. Since $\calA^{sa}$ is covered by $\{\calA^{sa}_\alpha: \calA_\alpha\in \calC(\calA)\}$, for any $a\in \calA$ there exists $\calA_\alpha\in \calC(\calA)$, s.t. $\mu(a^*a)=\mu|_{\calA_\alpha}(a^*a)\geq 0$, $\mu(1)=\mu|_{\calA_\alpha}(1)=1$. 
	Moreover, since by Lemma \ref{lemma about self-adjoint part}, $\calA^{sa}$ determines $\calC(\calA)$, and, in particular, determines each state space $S(\calA_\alpha)=S(\calA^{sa}_\alpha)$, every element of $QS(\calA)$ is uniquely defined by its values on $\calA^{sa}$.
	We can establish an isomorphism between $QS(\calA)$ and the set of quasi-states in the sense of Definition \ref{quasi-state} extending $\mu\in QS(\calA)$ from $\calA^{sa}$ to $\calA$ by the formula: $\mu(a+ib)=\mu(a)+i\mu(b)$, $\forall a,b\in \calA^{sa}$. Thus, we conclude that every element of $QS(\calA)$ can be thought of as a quasi-state in the sense of Definition (\ref{quasi-state}).
	
	Recall, that a quasi-state on a $C^*$-algebra $\calA$ is a map $\mu: \calA \rightarrow \C$, such that it is linear on all commutative subalgebras, satisfies $\mu(a+ib)=\mu(a)+i\mu(b)$ for all self-adjoint $a,b\in \calA$, $\mu(a^*a)\geq 0$ for all $a\in \calA$, and $\mu(1)=1$. If we restrict quasi-state on any unital commutative $C^*$-subalgebra $\calA_\alpha$, we obtain a positive linear functional s.t. $\mu(1)=1$. It follows from positivity, that $\mu_\alpha(f)\leq 1$, if $||f||\leq 1$ and $f\in \calA_\alpha^{sa}$ (since $1-f\geq 0$ in this case). Hence, $||\mu_\alpha||=1$, $\mu_\alpha$ is a state on $\calA_\alpha$, and $(\mu_\alpha)\in \prod S(\calA_\alpha)$. It is straightforward to check that $(\mu_\alpha)\in \varprojlim \{S(\calA_\alpha)\}$.
	
	Since $QS$ as a topological space is a projective limit of topological spaces $S(\calA_\alpha)$, the topology on it is defined as the weakest one, s.t. all projections (restriction maps in our case) are continuous. Each $S(\calA_\alpha)$ is equipped with the weak$^*$-topology, hence, by definition, $\mu_\alpha\rightarrow \mu_\alpha(f)$ for $f\in \calA_\alpha$ should be continuous.
	
	Convex structure on $QS(\calA)$ is inherited from $\prod S(\calA_\alpha)$, and, due to the fact, that $\calA^{sa}$ is covered by $\{\calA^{sa}_\alpha: \calA_\alpha\in \calC(\calA)\}$, it coincides with the natural (element-wise) one convex structure:
	$$
	(t \mu+ (1-t)\nu)(f) = t \mu_\alpha (f) + (1-t) \nu_\alpha (f)=t \mu(f)+ (1-t) \nu(f),\ \forall t\in[0,1]
	$$
	for $f\in \calA^{sa}_\alpha\subseteq \calA^{sa}$.
\end{proof}
Note, that it follows directly from the definition, that $QS(\calA)$ is a compact convex Hausdorff space.
\begin{proposition}
	$QS(\calA)$ is equipped with the weakest topology such that $\forall f\in \calA$, $\mu\rightarrow \mu(f)$ is a continuous functional.
\end{proposition}
\begin{proof}
	By the definition of the quasi-state space, for every unital commutative $C^*$-subalgebra $\calA_\alpha\subseteq \calA$, all functionals of the form $\mu_\alpha\rightarrow \mu_\alpha(f)$ for $f\in \calA_\alpha$ are continuous. Since $\calA^{sa}$ is covered by self-adjoint parts of elements of $\calC(\calA)$, for any $a\in \calA^{sa}$ there exists a commutative unital $C^*$-subalgebra $\calA_\alpha$ of $\calA$, s.t. $a\in \calA^{sa}_\alpha$. Hence $\mu\rightarrow \mu(a)=\mu|_{\calA_\alpha}(a)$ is continuous for all $f\in \calA^{sa}$. Since every $f$ can be represented as $f=a+ib$, where $a,b\in \calA^{sa}$, $a\in \calA^{sa}_\alpha$, $b\in \calA^{sa}_\beta$ for some commutative unital $C^*$-subalgebras $\calA_\alpha$, $\calA_\beta$ of $\calA$, the functional 
	$$
	\mu\rightarrow \mu(f)=\mu(a)+i\mu(b)=\mu|_{\calA_\alpha}(a)+ i\mu|_{\calA_\beta}(b)
	$$
	is continuous (since it is a linear combination of two continuous functionals) for every $f\in \calA$.
	
	If $\forall f\in \calA$, $\mu\rightarrow \mu(f)$ is continuous, then for any commutative unital $C^*$-subalgebra $\calA_\alpha$, $f\in \calA_\alpha\subseteq \calA$ implies $\mu|_{\calA_\alpha}\rightarrow \mu|_{\calA_\alpha}(f)=\mu(f)$ is continuous.
\end{proof}
\begin{corollary}
	\label{topologies on S and QS coincide} 
	Topology on $S(\calA)$, induced by the inclusion $S\subseteq QS(\calA)$, coincides with the weak$^*$-topology.
\end{corollary}
\begin{proof}
	$QS(\calA)$ is equipped with the weakest topology, such that $\forall f\in \calA$, $\mu\rightarrow \mu(f)$ is a continuous map. It is exactly the definition of the weak$^*$-topology on $S(\calA)$.
\end{proof}

Let us review the picture. 
\begin{itemize}
	\item Both $S(\calA)$ and $QS(\calA)$ are compact convex spaces for any unital $C^*$-algebra $\calA$. There is a natural inclusion $S(\calA)\subseteq QS(\calA)$, such that $S(\calA)$ is a compact convex subset of $QS(\calA)$.
	\item $QS(\calA)$ is a projective limit of $(\{S(\calA_\alpha)\}, \calR_\#)$ in the category $\bf CCH$ of compact convex spaces. In some sense this result is dual to the result of \cite{Heunen3}, where the injective limit of $(\{\calA_\alpha\}, \subseteq)$ is described in the category of partial $C^*$-algebras (see its definition there). We do not formalize this duality, since it is not clear how to axiomatically describe $QS$-spaces.
	\item Obviously, the state space $S(\calA)$ contains not less information than the corresponding quasi-state space $QS(\calA)$. The following picture illustrates the relation between the objects. The arrow $\rightsquigarrow$ means \textquotedblleft the right object is uniquely determined by the left one", $\leftrightsquigarrow$ means \textquotedblleft both the right and the left objects determine each other in the unique way".
	
	\begin{tikzcd}
		S(\calA)
		\arrow [leftrightsquigarrow] {r}  {}
		\arrow[hook]{d}{}
		& \calA^{sa}
		\arrow [squiggly]{r}{}
		& (\{\calA^{sa}_\alpha\}, \subseteq)
		\arrow [leftrightsquigarrow]{dl}{}
		\\
		QS(\calA)
		& (\{\calA_\alpha\}, \subseteq)
		\arrow [squiggly]{l}{}
	\end{tikzcd}
	
	In the case $S(\calA)=QS(\calA)$, all entities in the picture uniquely determine each other. It is not clear for the author, is there an arrow $QS(\calA)\rightsquigarrow (\{\calA_\alpha\}, \subseteq)$ or not in the general case (i.e. is it possible to recover the diagram of all unital commutative $C^*$-subalgebras knowing only the quasi-state space of an algebra?).
\end{itemize}

\section{Wasserstein distances for abelian $C^*$-subalgebras}
\sectionmark{Wasserstein distances for abelian $C^*$-subalgebras}

Let $\calA$ be a unital $C^*$-algebra. Define
\begin{itemize}
	\item $\calC(\calA):=(\{\calA_\alpha\}, \subseteq)$ as a diagram in $\bf ucC^*_{in}$ of all unital commutative $C^*$-subalgebras of $\calA$ ordered by inclusion,
	\item $\Spec(\calC(\calA)):=(\{\Spec(\calA_\alpha)\}, \calR)$ as a diagram in $\bf CH_{sur}$ of the Gelfand spectra, ordered by restriction maps: $R_{\beta,\alpha}: \Spec(\calA_\beta)\twoheadrightarrow \Spec(\calA_\alpha)$, $R_{\beta,\alpha}(\varphi_\beta)=\varphi_\alpha$, which are defined for any ordered pair $\calA_\alpha\subseteq \calA_\beta$ from $\calC(\calA)$. For shortness, we shall denote $\Spec(\calA_\alpha)$ as $X_\alpha$, thus, $\Spec(\calC(\calA)):=(\{X_\alpha\}, \calR)$,
	\item $S(\calC(\calA)):=(\{S(\calA_\alpha)\}, \calR_\#)$ as a diagram in $\bf CCH$ of the state spaces, ordered by restriction maps: $(R_{\beta,\alpha})_\#: S(\calA_\beta)\twoheadrightarrow S(\calA_\alpha)$, $(R_{\beta,\alpha})_\#(\mu_\beta)=\mu_\alpha$, which are defined for any ordered pair $\calA_\alpha\subseteq \calA_\beta$ from $\calC(\calA)$. As it was discussed earlier, $S(\calA_\alpha)\simeq\calP(X_\alpha)$. By abuse of notation, we shall denote by the same letter a state (functional) defined on $\calA_\alpha$ and its corresponding representation as a measure on $X_\alpha$. As it has been shown earlier, $(R_{\beta,\alpha})_\#(\mu_\beta)=\mu_\beta\circ R_{\beta,\alpha}^{-1}$.
\end{itemize}
We are going to use the theory of $\Lip$-seminorms on $C^*$-algebras (or, more generally, on order-unit spaces), which was developed by Rieffel (\cite{Rieffel1}, \cite{Rieffel2}, \cite{Rieffel3}). The idea is to define an abstract analogue of a Lipschitz gauge axiomatically. 
Let $\calA$ be a unital $C^*$-algebra, $L: \calA^{sa} \rightarrow [0,+\infty]$ be an abstract analogue of a Lipschitz gauge. Define the following notation:
\begin{itemize}
	\item $B(\calA):=\{a\in \calA^{sa}: L(a)<+\infty\}$
	\item $B_1(\calA):=\{a\in \calA^{sa}: L(a)=1\}$
	\item $N(\calA):=\{a\in \calA^{sa}: L(a)=0\}$
\end{itemize}
If it does not lead to a confusion, we shall use $B$, $B_1$, $N$ instead of $B(\calA)$, $B_1(\calA)$ and $N(\calA)$ respectively.

Let us provide the following definition.
\begin{definition}
	$L: \calA^{sa} \rightarrow [0,+\infty]$ is an \textbf{$L$-seminorm} on a unital $C^*$-algebra $\calA$ iff it is
	\begin{enumerate}
		\item absolutely homogeneous: $L(af)=|a|L(f)$, for any $a\in \R$, $f\in \calA^{sa}$,
		\item subadditive: $L(f+g)\leq L(f)+L(g)$, for any $f, g\in \calA^{sa}$,
		\item $L(1)=0$,
		\item $B(\calA)$ separates points in $S(\calA)$: for any two distinct states $\mu,\nu\in S(\calA)$ there exists $f\in B(\calA)$  s.t. $\mu(f)\neq \nu(f)$.
	\end{enumerate}
\end{definition}

\begin{remark}
	We are going to use the term $\Lip$-seminorm in the same sense as Rieffel do. Since the axioms above are weaker then the axioms of $\Lip$-seminorm, we gave another name for this object.
\end{remark}

\begin{lemma}
	\label{separation of points}
	The property \textquotedblleft$B(\calA)$ separates points in $S(\calA)$" is equivalent to \textquotedblleft$B(\calA)$ is a weak dense subspace of $\calA$", where weak density means $S(\calA)=\Sp_1^+(B(\calA)^*)$ ($B(\calA)$ is assumed to be equipped with the topology induced by the inclusion in $\calA$).
\end{lemma}
\begin{proof}
	If $B:=B(\calA)$ does not separate points, there are two $\mu, \nu \in S(\calA)$, $\mu\neq \nu$ s.t. $\mu(f)=\nu(f)$, $\forall f\in B$. Hence the restriction map $S(\calA)\rightarrow \Sp_1^+(B^*)$, $\mu\rightarrow \mu|_B$ is not injective, because different $\mu$ and $\nu$ has the same image.
	
	If $S(\calA)\neq \Sp_1^+(B^*)$ then $S(\calA)\rightarrow \Sp_1^+(B^*)$, $\mu\rightarrow \mu|_B$ is not injective (since it is surjective due to Hahn-Banach theorem, but not one-to-one). Hence there are two distinct points $\mu,\nu\in S(\calA)$ s.t. $\forall f\in B: \mu(f)=\nu(f)$, which contradicts with the separation of points. 
\end{proof}	

Sometimes it is useful to have a weaker definition of an abstract Lipschitz gauge:
\begin{definition}
	$L: \calA^{sa} \rightarrow [0,+\infty]$ is a \textbf{partially-defined $L$-seminorm} on a unital $C^*$-algebra $\calA$ if it is
	\begin{enumerate}
		\item absolutely homogeneous: $L(af)=|a|L(f)$, for any $a\in \R$, $f\in \calA^{sa}$,
		\item subadditive: $L(f+g)\leq L(f)+L(g)$, for any $f, g\in \calA^{sa}$,
		\item $L(1)=0$.
	\end{enumerate}
\end{definition}
Here we do not require any type of separation of points. Thus, it is possible for a partially-defined $L$-seminorm to be infinite everywhere except $1\cdot \R$. 
The important (but obvious) fact is that a restriction of any $L$-seminorm from $\calA$ to any unital $C^*$-subalgebra $\calA_\alpha$ of $\calA$ is a partially-defined $L$-seminorm on $\calA_\alpha$.

The following map $d_L:S(\calA)\times S(\calA) \rightarrow [0,+\infty]$
$$
d_L(\mu,\nu):=\sup\{|\mu(f)-\nu(f)|: L(f)\leq 1, f\in \calA^{sa}\}
$$
is called a \textbf{spectral distance} on $S(\calA)$. Let us review some known facts about this map.
Recall that a \textbf{pseudo-distance function} is the same thing as a distance function, except it may vanish on pairs of non-equal points.
\begin{proposition}
	\label{d_L is distance on S}
	If $L$ is a partially-defined seminorm on a unital $C^*$-algebra $\calA$, then $d_L$ is a $[0,+\infty]$-valued lower semi-continuous pseudo-distance function on $S(\calA)$. If $L$ is an $L$-seminorm, $d_L$ is $[0,+\infty]$-valued distance function on $S(\calA)$.
\end{proposition}
\begin{proof}
	Similar statements were proved by many authors (see, for example, \cite{Connes}, \cite{Rieffel1}). For convenience, we provide a proof here.
	\begin{itemize}
		\item Since $L(1)=0$, there is at least one $f$ s.t. $L(f)\leq 1$. For every such $f$, $F(\mu,\nu):=|\mu(f)-\nu(f)|$ is a continuous map from $S(\calA)\times S(\calA)$ to $[0,+\infty)$ (since every $\mu\rightarrow \mu(f)$ is continuous by definition of the weak$^*$-topology on $S(\calA$)). Hence $d_L$ is a lower semi-continuous map from $S(\calA)\times S(\calA)$ to $[0,+\infty]$ as a pointwise supremum of continuous maps.
		\item Obviously, $d_L(\mu,\mu)=0$. If $L$ is $L$-seminorm, for $\mu\neq \nu$, due to the weak density of $B$, $\mu|_B\neq \nu|_B$. Hence, there is an $f\in B$ s.t. $|\mu(f)-\nu(f)|>0$. We can scale this $f$ by some $\alpha>0$, s.t. $L(\alpha f)\leq 1$, and obtain that $d_L(\mu,\nu)\geq |\alpha||\mu(f)-\nu(f)|>0$. Thus, in the case of $L$-seminorm, $d_L(\mu,\nu)=0$ iff $\mu=\nu$.
		\item $d_L(\mu,\nu)=d_L(\nu,\mu)$ obviously.
		\item Check the triangle inequality:
		\begin{multline*}
		d_L(\mu,\nu)+d_L(\nu,\eta)=\\
		=\sup\{|\mu(f)-\nu(f)|: L(f)\leq 1, f\in \calA^{sa}\}
		+\sup\{|\eta(f)-\nu(f)|: L(f)\leq 1, f\in \calA^{sa}\}\geq\\
		\geq \sup\{|\mu(f)-\nu(f)|+|\eta(f)-\nu(f)|: L(f)\leq 1, f\in \calA^{sa}\}\geq\\
		\geq \sup\{|\mu(f)-\nu(f)+\nu(f)-\eta(f)|: L(f)\leq 1, f\in \calA^{sa}\}\geq\\
		\geq \sup\{|\mu(f)-\eta(f)|: L(f)\leq 1, f\in \calA^{sa}\}=d_L(\mu,\eta)
		\end{multline*}
	\end{itemize}
\end{proof}

The important result of Rieffel (\cite{Rieffel2}, Th. 4.1, Th 4.2) is that in the case $L$ is an $L$-seminorm,
$$
L_d(f):=\sup\left\{\frac{f(\mu)-f(\nu)}{d_L(\mu,\nu)}: \mu\neq \nu,\ \mu,\nu\in S(\calA)\right\}
$$
is also an $L$-seminorm, and it induces the same distance on $S(\calA)$ as $L$ does: $d_L=d_{L_d}$. Moreover, $L_{d_{L_d}}=L_d$. If $L$ is a lower semi-continuous $L$-seminorm: $\{f\in \calA^{sa}: L(f)\leq t\}$ is norm closed for one, hence every, $t\in (0,+\infty)$, then $L=L_d$. We do not require lower semi-continuity in the definition of $L$-seminorm, but it is clear, that we are always able to consider $L_d$ instead of the original $L$-seminorm $L$.

In the case $L$ is an $L$-seminorm, it is also true (\cite{Rieffel1}, Pr. 1.4), that topology generated by $d_L$ is not weaker than the weak$^*$-topology on $S(\calA)$. In the case of separable $\calA$, and, hence, metrizable $S(\calA)$, it is natural to ask these two topologies to coincide. The following definition is due to Rieffel (see \cite{Rieffel2}).
\begin{definition}
	An $L$-seminorm $L$ on a unital separable $C^*$-algebra $\calA$ is called \textbf{$\Lip$-seminorm} iff 
	\begin{enumerate}
		\item $L(f)=0 \iff f= a\cdot 1,\ a\in \R$,
		\item a distance function $d_L$ on $S(\calA)$ metricizes the weak$^*$-topology on $S(\calA)$.
	\end{enumerate}
\end{definition}
It is possible to formulate a criterion for an $L$-seminorm $L$ to be a $\Lip$-seminorm. Let $B:=B(\calA)$, $B_1:=B_1(\calA)$, $N:=N(\calA)$. 
Note that $B/N$ is a norm space equipped with the quotient norm $||\cdot||_{B/N}$ defined as  
$$
||f+N||_{B/N}:=\inf_{n\in N} ||f+n||_B,
$$
where $||\cdot||_B$ is a norm on $B$ inherited from the inclusion $B\subseteq\calA$.

\begin{proposition}[Criterion for $L$ to be a $\Lip$-seminorm, \cite{Rieffel1}, Th. 1.8]
	\label{criterion of Lip-seminorm}
	An $L$-seminorm $L$ on a unital separable $C^*$-algebra $\calA$ is a $\Lip$-seminorm iff $N=\R\cdot 1$ and the image of $B_1$ in $B/N$ is totally bounded w.r.t. the quotient norm $||\cdot||_{B/N}$.
\end{proposition}
Now we consider a unital commutative $C^*$-subalgebra $\calA_\alpha\subseteq \calA$.
Let us define $d_{L, \alpha}: S(\calA_\alpha)\times S(\calA_\alpha) \rightarrow [0,\infty]$ as follows:
$$
d_{L, \alpha}(\mu_\alpha,\nu_\alpha):=\sup\{|\mu(f)-\nu(f)|: f\in \calA^{sa}_\alpha,\ L(f)\leq 1\}
$$
\begin{proposition}
	\label{d_L,alpha is distance}
	For any $C^*$-algebra $\calA$ and any partially-defined $L$-seminorm $L$ on $\calA$,
	$d_{L, \alpha}$	is a lower semi-continuous $[0,+\infty]$-valued pseudo-distance function. If $B_\alpha:=B(\calA)\cap \calA_\alpha$ separates points in $S(\calA_\alpha)$ (or, equivalently, restriction of $L$ to $\calA_\alpha$ is an $L$-seminorm), then it is a measurable $[0,+\infty]$-valued distance function.
\end{proposition}
\begin{proof}
	\begin{itemize}
		\item Since $L(1)=0$, and $\calA_\alpha$ is a unital subalgebra, there exists at least one $f\in \calA^{sa}_\alpha$ s.t. $L(f)\leq 1$. For every such $f$, $F(\mu_\alpha,\nu_\alpha):=|\mu_\alpha(f)-\nu_\alpha(f)|$ is a continuous map from $S(\calA_\alpha)\times S(\calA_\alpha)$ to $[0,+\infty)$. Hence, $d_{L, \alpha}$ is a lower semi-continuous map from $S(\calA)\times S(\calA_\alpha)$ to $[0,+\infty]$ as a pointwise supremum of continuous maps.
		\item Obviously, $d_{L, \alpha}(\mu_\alpha,\mu_\alpha)=0$.
		\item If $B_\alpha$ is weak dense in $\calA_\alpha$, for every $\mu_\alpha\neq \nu_\alpha$, $\mu_\alpha|_B\neq \nu_\alpha|_B$. Hence, there is an $f\in B$ s.t. $|\mu_\alpha(f)-\nu_\alpha(f)|>0$. We can scale this $f$ by some $t>0$, s.t. $L(t f)\leq 1$, and obtain that $d_{L, \alpha}(\mu_\alpha,\nu_\alpha)\geq |t||\mu_\alpha(f)-\nu_\alpha(f)|>0$. Thus, $d_{L, \alpha}(\mu_\alpha,\nu_\alpha)=0$ iff $\mu_\alpha=\nu_\alpha$.
		\item $d_{L, \alpha}(\mu_\alpha,\nu_\alpha)=d_{L, \alpha}(\nu_\alpha,\mu_\alpha)$ obviously.
		\item Check the triangle inequality:
		\begin{multline*}
		d_{L, \alpha}(\mu_\alpha,\nu_\alpha)+d_{L, \alpha}(\nu_\alpha,\eta_\alpha)=\\
		=\sup\{|\mu_\alpha(f)-\nu_\alpha(f)|: L(f)\leq 1, f\in \calA^{sa}_\alpha\}
		+\sup\{|\eta_\alpha(f)-\nu_\alpha(f)|: L(f)\leq 1, f\in \calA^{sa}_\alpha\}\geq\\
		\geq \sup\{|\mu_\alpha(f)-\nu_\alpha(f)|+|\eta(f)-\nu_\alpha(f)|: L(f)\leq 1, f\in \calA^{sa}_\alpha\}\geq\\
		\geq \sup\{|\mu_\alpha(f)-\nu_\alpha(f)+\nu_\alpha(f)-\eta_\alpha(f)|: L(f)\leq 1, f\in \calA^{sa}_\alpha\}\geq\\
		\geq \sup\{|\mu_\alpha(f)-\eta_\alpha(f)|: L(f)\leq 1, f\in \calA^{sa}_\alpha\}=d_{L, \alpha}(\mu_\alpha,\eta_\alpha)
		\end{multline*}
	\end{itemize}
\end{proof}	
Note, that the restriction of $d_{L, \alpha}$ on $X_\alpha\subseteq S(\calA_\alpha)$ remains to be a lower semi-continuous $[0,+\infty]$-valued pseudo-distance function.

\begin{lemma}
	\label{Lip seminorm on subalgebra}
	If $B\cap \calA_\alpha$ separates points in $S(\calA_\alpha)$, where $B:=\{f\in \calA^{sa}: L(f)< \infty\}$, $\calA_\alpha$ is a unital commutative $C^*$-subalgebra of $\calA$ (equivalently, the restriction of $L$ to $\calA_\alpha$ is an $L$-seminorm), and $L$ is a $\Lip$-seminorm on $\calA$, then the restriction of $L$ to $\calA_\alpha$ is also a $\Lip$-seminorm.
\end{lemma}
\begin{proof}
	Since the restriction of $L$ to $\calA_\alpha$ is an $L$-seminorm, it remains to check, that the distance
	$d_{L, \alpha}$ metricizes the weak$^*$-topology on $\calA_\alpha$. Let us use the criterion from Proposition \ref{criterion of Lip-seminorm}. Since $N=\R\cdot 1$, $N\subseteq \calA_\alpha$, $C^*$-norm on $\calA_\alpha$ coincides with the norm inherited from $\calA$, and $(B\cap \calA_\alpha)/(N\cap \calA_\alpha)= (B\cap \calA_\alpha)/N\subseteq B/N$, it follows that the norms $||\cdot||_{B/N}=||\cdot||_{(B\cap \calA_\alpha)/(N\cap \calA_\alpha)}$ on their common domain of definition. The image of $B_1\cap \calA_\alpha$ in $(B\cap \calA_\alpha)/(N\cap \calA_\alpha)$ coincides with the image of $B_1\cap \calA_\alpha$ in $B/N$, and, in fact, it is a subset of a totally bounded image of $B_1$ in $B/N$. Hence, it is totally bounded.
\end{proof}

Let us define $W_{p, \alpha}: S(\calA_\alpha)\times S(\calA_\alpha) \rightarrow [0,\infty]$ for $p\in [1,+\infty)$ by the formula:
$$
W_{p,\alpha}(\mu_\alpha,\nu_\alpha):=\inf\left\{\left(\int d_{L, \alpha}^p(x,y) d\pi\right)^\frac{1}{p}: \pi\in \calP(X_\alpha\times X_\alpha),\ \mathrm{Pr}(\pi)=(\mu_\alpha,\nu_\alpha)\right\}
$$
Here $\calP(X_\alpha\times X_\alpha)$ is the set of all Radon probability measures on $X_\alpha\times X_\alpha$ (recall, that each $X_\alpha$ is equipped with the weak$^*$-topology, and all $\sigma$-algebras are supposed to be Borel), $\mathrm{Pr}(\pi):=((\Pr_1)_\#(\pi), (\Pr_2)_\#(\pi))$.

In the formulation of the next statement we shall use the following notion, introduced in \cite{Rieffel2}.
\begin{definition}{(\cite{Rieffel2}, Definition 4.5)}
	Let $L$ be a $\Lip$-seminorm on a $C^*$-algebra $\calA$, $B:=\{a\in \calA^{sa}: L(a)< \infty\}$, $B_1:=\{a\in \calA^{sa}: L(a)\leq 1\}$, $\bar{B}$, $\bar{B}_1$ be their closures w.r.t. $C^*$-norm on $\calA$. The \textbf{closure} of $L$, $\bar{L}: \bar{B} \rightarrow [0, +\infty]$ is defined as
	$$
	\bar{L}(a):=\inf\{t\in [0,+\infty]: \exists b\in \bar{B}_1 \text{ s.t. } a=t\cdot b \}
	$$
\end{definition}

\begin{proposition}
	\label{W_1,alpha=d_L,alpha}
	In the case
	\begin{enumerate}
		\item $\calA$ is a separable unital $C^*$-algebra,
		\item $L$ is a $\Lip$-seminorm on $\calA$,
		\item $B\cap \calA_\alpha$ is dense in $\calA_\alpha$,
		\item for $L_\alpha:=L|_{\calA_\alpha}$ the inequality
		$$
		\bar{L}_\alpha(\sup(f,g))\leq \max(\bar{L}_\alpha(f), \bar{L}_\alpha(g))
		$$
		is satisfied for all $f,g\in \bar{B}_1\cap \calA_\alpha$,
	\end{enumerate}
	we obtain a valid equality:
	$$
	d_{L, \alpha}(\mu_\alpha,\nu_\alpha)=\inf\left\{\int d_{L, \alpha}(x,y) d\pi: \pi\in \calP(X_\alpha\times X_\alpha),\  \pr(\pi)=(\mu_\alpha,\nu_\alpha)\right\}=:W_{1,\alpha}(\mu_\alpha,\nu_\alpha)
	$$
	for all $\mu_\alpha,\nu_\alpha\in S(\calA_\alpha)$.
\end{proposition}
\begin{proof}
	Since $S(\calA_\alpha)$ is a Bauer simplex, each element $\mu_\alpha\in S(\calA_\alpha)$ has a unique representation as a Radon probability measure on $\partial_e(S(\calA_\alpha))\simeq X_\alpha$, i.e. $\tilde{\mu}_\alpha\in \calP(X_\alpha)$, $\mathrm{bar}(\tilde{\mu}_\alpha)=\mu_\alpha$.
	By Lemma \ref{Lip seminorm on subalgebra}, the restriction of $L$ on $\calA_\alpha$ is a $\Lip$-seminorm, thus $d_L$ metricizes the weak$^*$-topology on $S(\calA)$.
	
	Due to the inequality from the hypothesis of the Lemma and density of $B\cap \calA_\alpha$, by Theorem 8.1 of \cite{Rieffel2}, $L_\alpha:=L|_{\calA_\alpha}$ coincides with the Lipschitz seminorm induced by the distance $d_L$ on $X_\alpha$:
	$$
	L_\alpha(f)=L^e_{d_L}(f):=\sup\left\{\frac{f(x)-f(y)}{d_L(x,y)}: x\neq y,\ x,y\in X_\alpha\right\},\ \forall f\in \calA_\alpha^{sa} 
	$$
	Hence we can use the Kantorovich-Rubinstein duality (see \cite{BogKol}) to obtain:
	\begin{multline*}
	W_{1,\alpha}(\mu_\alpha,\nu_\alpha):=\inf\left\{\int d_{L, \alpha}(x,y) d\pi: \pi\in \calP(X_\alpha\times X_\alpha),\ \pr(\pi)=(\tilde{\mu}_\alpha,\tilde{\nu}_\alpha)\right\}=\\
	=\sup\{\mu_\alpha(f)-\nu_\alpha(f): f\in \calA_\alpha^{sa},\ L^e_{d_L}(f)\leq 1\}=\\
	=\sup\{\mu_\alpha(f)-\nu_\alpha(f): f\in \calA_\alpha^{sa},\ L(f)\leq 1\}=:d_{L, \alpha}(\mu_\alpha,\nu_\alpha)
	\end{multline*}
\end{proof}

\begin{proposition}
	\label{W_p,alpha is distance}
	For any $C^*$-algebra $\calA$ and any partially-defined $L$-seminorm $L$ on $\calA$,
	$W_{p, \alpha}$	is a $[0,+\infty]$-valued pseudo-distance function on $S(\calA_\alpha)$. If $B_\alpha:=B\cap \calA_\alpha$ separates points in $\calA_\alpha$ (or, equivalently, the restriction of $L$ to $\calA_\alpha$ is an $L$-seminorm), then $W_{p, \alpha}$ is a $[0,+\infty]$-valued distance function.
\end{proposition}
\begin{proof}
	By a slight abuse of notation, we denote as $\mu_\alpha$ both the linear functional (state) $\mu$ restricted to $\calA_\alpha$ and the measure on $X_\alpha$ that corresponds to this state. By the Rietz representation theorem, it is a Radon probability measure. By $\Pi(\mu_\alpha,\nu_\alpha)$ we denote the set of all Radon probability measures on $X_\alpha\times X_\alpha$ with the marginals $\mu_\alpha$ and $\nu_\alpha$.
	\begin{itemize}
		\item $W_{p, \alpha}(\mu_\alpha,\mu_\alpha)=0$, because $d^p_{L,\alpha}(x,x)=0$, and $(\Id,\Id)_\#\mu_\alpha$ is an element of $\Pi(\mu_\alpha, \mu_\alpha)$, which is concentrated on the diagonal set $\Diag:=\{(x,x):x\in X_\alpha\}$.
		\item If $B_\alpha$ is weak dense in $B$, $d^p_{L,\alpha}$ is a $[0,+\infty]$-valued distance function. Hence $d^p_{L,\alpha}(x,x)=0$ iff $x=y$. Since $\Pi(\mu_\alpha, \nu_\alpha)$ contains no measures $\pi$ s.t. $\pi(\Diag)=1$, when $\mu_\alpha\neq \nu_\alpha$, it follows that $W_{p, \alpha}(\mu_\alpha,\nu_\alpha)>0$ in this case.
		\item $W_{p, \alpha}(\mu_\alpha,\nu_\alpha)=W_{p, \alpha}(\nu_\alpha,\mu_\alpha)$, because $\forall \pi\in \Pi(\mu_\alpha,\nu_\alpha)$ there exists $\pi^T\in \Pi(\nu_\alpha, \mu_\alpha)$ defined as $\int f(x,y) d\pi=\int f(y,x) d\pi^T$ for all measurable functions $f$ on $X_\alpha\times X_\alpha$, and because $d^p_{L,\alpha}$ is Borel measurable and symmetric.
		\item By Lemma 1.1.6 from \cite{BogKol} for any two Radon measures $\pi_1, \pi_2$ on $X_\alpha\times X_\alpha$ s.t. the second marginal of the first measure coincides with the first marginal of the second measure, there exists a Radon measure on $X_\alpha\times X_\alpha\times X_\alpha$, such that its projection on the first two factors is equal to $\pi_1$, and the projection on the last two factors is equal to $\pi_2$. Let for any $n\in \N$, $\pi_1^n\in \Pi(\mu_\alpha,\nu_\alpha)$ s.t. $W_{p,\alpha}(\mu_\alpha,\nu_\alpha)\geq \left(\int d^p_{L,\alpha} \pi_1^n\right)^{\frac{1}{p}}-\frac{1}{n}$, $\pi_2^n\in \Pi(\nu_\alpha,\eta_\alpha)$ s.t. $W_{p,\alpha}(\nu_\alpha,\eta_\alpha)\geq \left(\int d^p_{L,\alpha} \pi_2^n\right)^{\frac{1}{p}}-\frac{1}{n}$, $\gamma^n$ be a measure in $\Pi(\mu_\alpha,\nu_\alpha,\eta_\alpha)$ as in the described Lemma with projections $\pi^n_1$, $\pi^n_2$.
		Then
		\begin{multline*}
		W_{p, \alpha}(\mu_\alpha,\eta_\alpha)\leq \left(\int d^p_{L,\alpha}(x,z) d\mathrm{Pr}^{1,3}_{\#}(\gamma^n)\right)^{\frac{1}{p}}=\left(\int d^p_{L,\alpha}(x,z) d\gamma^n\right)^{\frac{1}{p}}\leq\\
		\leq \left(\int (d_{L,\alpha}(x,y)+d_{L,\alpha}(y,z))^p d\gamma^n\right)^{\frac{1}{p}}
		\leq \left(\int (d_{L,\alpha}(x,y) d\gamma^n\right)^{\frac{1}{p}}+\left(\int (d_{L,\alpha}(y,z) d\gamma^n\right)^{\frac{1}{p}}=\\
		=\left(\int d^p_{L,\alpha} \pi_1^n\right)^{\frac{1}{p}}+\left(\int d^p_{L,\alpha} \pi_2^n\right)^{\frac{1}{p}}\leq W_{p,\alpha}(\mu_\alpha,\nu_\alpha)+W_{p,\alpha}(\nu_\alpha,\eta_\alpha)+\frac{2}{n}
		\end{multline*}
		Taking a limit as $n\rightarrow \infty$ in both parts of the inequality, we obtain the required statement.
	\end{itemize}	
\end{proof}	

\begin{proposition}
	\label{W_beta>=W_alpha}
	Let $L$ be a partially-defined $L$-seminorm on a unital $C^*$-algebra $\calA$.
	For any $\mu,\nu\in QS(\calA)$ and any pair $\calA_\alpha\subseteq \calA_\beta$ of unital commutative $C^*$-subalgebras of $\calA$, it is true that
	$$
	W_{p,\beta}(\mu_\beta, \nu_\beta)\geq W_{p, \alpha}(\mu_\alpha,\nu_\alpha),
	$$
\end{proposition}
\begin{proof}
	Note, that $d_{L,\beta}(\mu_\beta,\nu_\beta)\geq d_{L,\alpha}(\mu_\alpha,\nu_\alpha)$ as follows directly from the definition of $d_{L,\alpha}$.
	Since the restriction map $R_{\beta,\alpha}: X_\beta \twoheadrightarrow X_\alpha$ is surjective (as follows from Gelfand duality), the direct product map $(R\times R)_{\beta,\alpha}: X_\beta\times X_\beta \twoheadrightarrow X_\alpha\times X_\alpha$, $(R\times R)_{\beta,\alpha}(x_\beta,y_\beta)=(R_{\beta,\alpha}(x_\beta), R_{\beta,\alpha}(y_\beta))$, and the corresponding pushforward of measures $((R\times R)_{\beta,\alpha})_\#: \calP(X_\beta\times X_\beta)\twoheadrightarrow \calP(X_\alpha\times X_\alpha)$ are also surjective. Since projection operator commutes with the restriction maps, $((R\times R)_{\beta,\alpha})_\#: \Pi(\mu_\beta, \nu_\beta)\twoheadrightarrow \Pi(\mu_\alpha, \nu_\alpha)$ is well-defined and surjective too. Thus, every $\pi_\alpha\in \calP(\mu_\alpha, \nu_\alpha)$ is equal to $((R\times R)_{\beta,\alpha})_\#\pi_\beta$ for some 
	$\pi_\beta\in \calP(\mu_\beta, \nu_\beta)$. It follows that $\int d_{L,\alpha}d\pi_\alpha=\int d_{L,\alpha}\circ (R\times R)_{\beta,\alpha}d\pi_\beta$. But 
	$$
	d_{L,\beta}(\mu_\beta,\nu_\beta)\geq d_{L,\alpha}(\mu_\alpha,\nu_\alpha)=(d_{L,\alpha}\circ (R\times R)_{\beta,\alpha})(\mu_\beta,\nu_\beta)
	$$
	and, hence, $\int d_{L,\alpha}d\pi_\alpha\leq \int d_{L,\beta}d\pi_\beta$. Thus, $W_{p,\beta}(\mu_\beta, \nu_\beta)\geq W_{p, \alpha} (\mu_\alpha,\nu_\alpha)$.
\end{proof}
\begin{proposition}
	\label{W_q>=W_p}
	Let $L$ be a partially-defined $L$-seminorm on a unital $C^*$-algebra $\calA$, $\calA_\alpha\subseteq \calA$ be a unital commutative $C^*$-subalgebra. Then
	$W_{p,\alpha}\leq W_{q,\alpha}$ if $p\leq q$.
\end{proposition}
\begin{proof}
 By definition, $W_{p,\alpha}(\mu_\alpha,\nu_\alpha)=\inf \{||d_{p,\alpha}||_{L^p(\pi)}: \pi\in \Pi(\mu_\alpha,\nu_\alpha)\}$.
	The statement follows directly from the fact, that $||\cdot||_{L^p(\pi)}\leq ||\cdot||_{L^q(\pi)}$ for $p\leq q$ and any probability measure $\pi$.
\end{proof}

\begin{proposition}
	Let $\calA$ be a separable unital $C^*$-algebra, $L$ be a $\Lip$-seminorm, and $B\cap \calA_\alpha$ separates points in  $S(\calA_\alpha)$, where $B:=\{f\in \calA^{sa}: L(f)< \infty\}$, $\calA_\alpha$ is a unital commutative $C^*$-subalgebra of $\calA$ (equivalently, the restriction of $L$ to $\calA_\alpha$ is an $L$-seminorm). Then $W_{p,\alpha}$ metricizes the weak$^*$-topology on $S(\calA_\alpha)$ for every $p\in [1,\infty)$.
\end{proposition}
\begin{proof}
	By Proposition \ref{Lip seminorm on subalgebra}, $L$, being restricted to $\calA_\alpha$, remains to be a $\Lip$-seminorm. Hence (by definition of a $\Lip$-seminorm) $d_{L,\alpha}$ metricizes the weak$^*$-topology on $S(\calA_\alpha)$.
	
	By definition, $W_{p,\alpha}^p(\mu_\alpha,\nu_\alpha):=\inf\left\{\int d_{L, \alpha}^p d\pi: \pi\in \Pi(\mu_\alpha,\nu_\alpha)\right\}$ is the $L^p$-Wasserstein distance between measures $\mu_\alpha$ and $\nu_\alpha$ on a compact metric space $(X_\alpha, d_{L,\alpha})$. Hence, by the standard theory of Wasserstein distances (see any of \cite{AmbGigli1}, \cite{BogKol}, \cite{Villani1}), it metricizes the weak$^*$-topology on $\calP(X_\alpha)=S(\calA_\alpha)$.
\end{proof}

\section{Projective $L^p$-Wasserstein distances}
\sectionmark{Projective $L^p$-Wasserstein distances}

Let $\calA$ be a unital $C^*$-algebra. Define $\calC(\calA):=(\{\calA_\alpha\}, \subseteq)$, $Spec(\calC(\calA)):=(\{\Spec(\calA_\alpha)\}, \calR)$, $S(\calC(\calA)):=(\{S(\calA_\alpha)\}, \calR_\#)$ as in the previous section.

We are going to construct a distance function on the quasi-state space of $\calA$ using distance functions on the state spaces of its unital commutative subalgebras. As we have already seen, distance functions $W_{p,\alpha}$ on $S(\calA_\alpha)$ are well-behaved (they are actually distances and metricize the weak$^*$-topologies) in the case $L$ is finite on a large enough subspace of $\calA_\alpha$, or, more precisely, if $B\subseteq \calA_\alpha$ separates points in $S(\calA_\alpha)$ (here $B:=\{a\in \calA^{sa}: L(a)<\infty\}$ as earlier). It motivates the following definition.

\begin{definition}
	An $L$-seminorm $L$ on a unital $C^*$-algebra $\calA$ is called \textbf{solid} iff $B\cap \calA_\alpha$ separates points in $S(\calA_\alpha)$ for any \textit{maximal} unital commutative subalgebra $\calA_\alpha$ of $\calA$.
\end{definition}
In other words, we ask an $L$-seminorm to remain to be an $L$-seminorm, when it is restricted to a maximal subalgebra. It appears (as Proposition \ref{Lip seminorm on subalgebra} states), that in this case a $\Lip$-seminorm remains to be a $\Lip$-seminorm when restricted to a maximal subalgebra.
\begin{example}
	If $L$-seminorm $L$ on $\calA$ has only finite values, it is solid.
\end{example}

\begin{remark}
	As it was shown in Lemma \ref{lemma about self-adjoint part}, for every unital $C^*$-algebra $\calA$, its self-adjoint part $\calA^{sa}$ is covered by self-adjoint parts of unital commutative $C^*$-subalgebras of $\calA$. Let us fix some $f\in \calA^{sa}$. Commutative unital $C^*$-subalgebras containing $f$ constitute a poset w.r.t. inclusion. Each chain of this poset has an upper bound (union of the corresponding subalgebras). Using Zorn's lemma, we conclude, that $f$ is a self-adjoint element of some maximal commutative unital $C^*$-subalgebra.
	Hence, the self-adjoint part of a unital $C^*$-algebra is covered by the self-adjoint parts of its maximal unital commutative $C^*$-subalgebras.
\end{remark}
As follows from the results of the previous section, if $L$ is a solid $L$-seminorm on $\calA$, then for every maximal unital commutative $C^*$-subalgebra $\calA_\alpha$, $W_{p,\alpha}$ is a $[0, +\infty]$-valued distance function. If, moreover, $L$ is a $\Lip$-seminorm (which is possible only if $\calA$ is separable), $W_{p,\alpha}$ metricizes the weak$^*$-topology on $S(\calA_\alpha)$.

\begin{proposition}
	\label{solid L separates points in QS}
	If $L$ is a solid $L$-seminorm on $\calA$, $B(\calA)$ separates points in $QS(\calA)$.
\end{proposition}
\begin{proof}
	Let $\mu,\nu\in QS(\calA)$, $\mu\neq \nu$, but $\mu|_B=\nu|_B$. It follows that $\mu|_{B\cap\calA_\alpha}=\nu|_{B\cap\calA_\alpha}$ for every maximal commutative unital $C^*$-subalgebra $\calA_\alpha$. Since $L$ is solid, $\mu|_{\calA_\alpha}=\nu|_{\calA_\alpha}$. Since maximal unital commutative subalgebras of $\calA$ covers $\calA^{sa}$, it follows that $\mu=\nu$, which contradicts the assumption.
\end{proof}
This fact allows us to extend distance function $d_L$ from the state space $S(\calA)$ to the quasi-state space $QS(\calA)$.

\begin{definition}
	Let $\calA$ be a unital $C^*$-algebra, $L$ be a partially-defined $L$-seminorm on $\calA$. Then $d_L$: $QS(\calA)\times QS(\calA) \rightarrow [0,\infty]$ is defined by the formula:
	$$
	d_L(\mu,\nu):=\sup\{|\mu(f)-\nu(f)|: L(f)\leq 1, f\in \calA^{sa}\}
	$$
\end{definition}
\begin{proposition}
	\label{d_L is distance on QS}
	$d_L$ is a $[0,+\infty]$-valued lower semi-continuous pseudo-distance function on $QS(\calA)$. If $L$ is a solid $L$-seminorm, $d_L$ is a $[0,+\infty]$-valued lower semi-continuous distance function on $QS(\calA)$.
\end{proposition}
\begin{proof}
 	The proof is analogous to the proof of Proposition \ref{d_L is distance on S}. We only provide an argument for the lower semi-continuity of $d_L$ on the new domain of definition. The rest part of the proof can be copy-pasted directly from the proof of Proposition \ref{d_L is distance on S}.
 	
 	Note, that $F(\mu,\nu):=|\mu(f)-\nu(f)|$ is a continuous map from $QS(\calA)\times QS(\calA)$ to $[0,+\infty)$, because every $\mu\rightarrow \mu(f)$ is continuous by the definition of the projective topology on $QS(\calA)$. Hence $d_L$ is a lower semi-continuous map from $QS(\calA)\times QS(\calA)$ to $[0,+\infty]$ as a pointwise supremum of continuous maps.
 	In case $L$ is solid, by Proposition \ref{solid L separates points in QS}, $B(\calA)$ separates points in $QS(\calA)$, and, hence, $d_L(\mu,\nu)=0$ implies $\mu=\nu$ (as in \ref{d_L is distance on S}).
\end{proof}

\begin{definition}
	For any $p\in [1,\infty]$ define the \textbf{projective $L^p$-Wasserstein distance} $\arWp{p}: QS(\calA)\times QS(\calA) \rightarrow [0,\infty]$ by the formula:
	$$
	\arWp{p}(\mu,\nu):=\sup_\alpha\{W_{p,\alpha}(\mu_\alpha,\nu_\alpha)\}
	$$
	where $\mu_\alpha:=\mu|_{\calA_\alpha}$.
\end{definition}
\begin{proposition}
	\label{W_proj is distance on QS}
	For any unital $C^*$-algebra $\calA$ and any partially-defined $L$-seminorm $L$, $\arWp{p}$ is a $[0,+\infty]$-valued pseudo-distance function. If $L$ is a solid $L$-seminorm, then $\arWp{p}$ is a $[0,+\infty]$-valued distance function.
\end{proposition}
\begin{proof}
	According to Proposition \ref{W_beta>=W_alpha}, for any two $\mu,\nu\in QS(\calA)$, $
	W_{p,\beta}(\mu_\beta, \nu_\beta)\geq W_{p, \alpha} (\mu_\alpha,\nu_\alpha)
	$ if $\calA_\alpha\subseteq \calA_\beta$. It follows that 
	$$
	\arWp{p}(\mu,\nu)=\sup\{W_{p,\beta}(\mu_\beta, \nu_\beta): \calA_\beta \text{ is a maximal element of } (\{\calA_\alpha\},\subseteq)\}
	$$
	In other words, we may take supremum over the set consisted of only maximal subalgebras. Since for each of these subalgebras, $W_{p,\alpha}$ is a $[0,+\infty]$-valued pseudo-distance function on $S(\calA_\alpha)$ (by Proposition \ref{W_p,alpha is distance}), we may think of it as a $[0,+\infty]$-valued pseudo-distance function on $S(\calA)$: $W_{p,\alpha}(\mu,\nu):=W_{p,\alpha}(\mu_\alpha,\nu_\alpha)$. It is straightforward to check that $\arWp{p}$ is a $[0,+\infty]$-valued pseudo-distance function.
	
	If $L$ is solid, by Proposition \ref{solid L separates points in QS}, $B(\calA)$ separates points in $QS(\calA)$. Hence, $\mu\neq\nu \implies \mu(f)\neq \nu(f)$ for some $f\in \calA$. Since a self-adjoint part of any $C^*$-algebra is covered by its maximal unital commutative subalgebras (and states of $C^*$-algebra are determined by their values on self-adjoint elements), it follows that $\mu\neq\nu \implies \mu_\alpha\neq \nu_\alpha$ for some maximal $\calA_\alpha\in \calC(\calA)$. Hence $\arWp{p}(\mu,\nu)>0$ iff $\mu\neq \nu$, and $\arWp{p}$ is a $[0,+\infty]$-valued distance function.
\end{proof}

\begin{definition}
	A $\Lip$-seminorm $L$ on a unital $C^*$-algebra $\calA$ is called \textbf{Wasserstein-compatible} iff for any \textit{maximal} unital commutative $C^*$-subalgebra $\calA_\alpha$ of $\calA$
	\begin{enumerate}
		\item $B\cap \calA_\alpha$ is dense in $\calA_\alpha$,
		\item for the closure of $L_\alpha:=L|_{\calA_\alpha}$, $\bar{L}_\alpha$, the inequality
		$$
		\bar{L}_\alpha(\sup(f,g))\leq \max(\bar{L}_\alpha(f), \bar{L}_\alpha(g))
		$$
		is satisfied for all $f,g\in \bar{B}_1\cap \calA_\alpha$. Here $B_1:=\{a\in\calA^{sa}: L(a)\leq 1\}$.
	\end{enumerate}
\end{definition}
Since a dense subspace of a Banach space separates points in its dual, every Wasserstein-compatible $\Lip$-seminorm is a solid $\Lip$-seminorm.
\begin{remark}
	If $\Lip$-seminorm $L$ on $\calA$ has only finite values and
	$$
	L(\sup(f,g))\leq \max(L(f), L(g)),\ \forall f,g\in \calA^{sa},
	$$
	then it is Wasserstein-compatible. More generally, if this inequality is satisfied, $B:=\{a\in\calA^{sa}: L(a)<\infty\}$ is closed under finite lattice operations, and each $B\cap \calA_\alpha$ is dense in $\calA_\alpha$ for any maximal unital commutative $C^*$-subalgebra $\calA_\alpha$ of $\calA$, then $L$ is Wasserstein-compatible. It follows from Corollary 8.3 of \cite{Rieffel2} and closedness of $C^*$-subalgebras.
\end{remark}
As we shall see later, if $\Lip$-seminorm is Wasserstein-compatible, the projective Wasserstein distances metricize the weak$^*$-topology of the state space. Unfortunately, this assumption seems to be too strong, and there are no known non-trivial examples of Wasserstein-compatible $\Lip$-seminorms.
\begin{proposition}
	\label{d_L=W_1_projective}
	For any unital separable $C^*$-algebra $\calA$ and any Wasserstein-compatible $\Lip$-seminorm $L$ on $\calA$, $d_L=\arWp{1}$ on $QS(\calA)$.
\end{proposition}
\begin{proof}
	Since self-adjoint part of every $C^*$-algebra is covered by its maximal unital commutative subalgebras,
	\begin{multline*}
	d_L(\mu,\nu):=\sup\{|\mu(f)-\nu(f)|:L(f)\leq 1,\ f\in \calA^{sa}\}=\\
	=\sup_\alpha\left\{\sup\{|\mu(f)-\nu(f)|:L(f)\leq 1,\ f\in \calA_\alpha^{sa}\}\right\},\ \forall \mu,\nu\in QS(\calA),
	\end{multline*}
	where $\alpha$ parametrizes the set of all maximal unital commutative $C^*$-subalgebras of $\calA$. It follows by the definition of Wasserstein-compatible $\Lip$-seminorm, that all assumptions of Proposition \ref{W_1,alpha=d_L,alpha} are satisfied, hence $W_{1,\alpha}=d_{L,\alpha}$ for every such $\alpha$. Thus
	\begin{multline*}
	\sup_\alpha\left\{\sup\{|\mu(f)-\nu(f)|:L(f)\leq 1,\ f\in \calA^{sa}_\alpha\}\right\}=\\
	=\sup_\alpha\left\{\sup\{|\mu_\alpha(f)-\nu_\alpha(f)|:L(f)\leq 1,\ f\in \calA^{sa}_\alpha\}\right\}=\\
	=\sup_\alpha \left\{ d_{L,\alpha}(\mu_\alpha,\nu_\alpha) \right\}=\sup_\alpha \left\{W_{1,\alpha}(\mu_\alpha,\nu_\alpha)\right\}=:\arWp{1}(\mu,\nu)
	\end{multline*}
	where $\mu_\alpha=\mu|_{\calA_\alpha}\in S(\calA_\alpha)$ is a restriction of a quasi-state to a subalgebra.
\end{proof}

Let us define \textbf{diameters} of the quasi-state space and the state space of $\calA$ as follows.
\begin{eqnarray}
&\Diam(QS(\calA)):=\sup\{d_L(\mu,\nu): \mu,\nu\in QS(\calA)\}\\
&\Diam(S(\calA)):=\sup\{d_L(\mu,\nu): \mu,\nu\in S(\calA)\}
\end{eqnarray}
Here $\calA$ is a unital $C^*$-algebra, $L$ is an $L$-seminorm. It is clear, that in the case of $\Lip$-seminorm $L$, $\Diam(S(\calA))<\infty$.
Let
\begin{equation}
\Diam(S(\calA_\alpha)):=\sup\{d_{L,\alpha}(\mu_\alpha,\nu_\alpha): \mu_\alpha,\nu_\alpha\in S(\calA_\alpha)\}
\end{equation} 

\begin{proposition}
	\label{Diam>=Diam_alpha}
	Let $\calA$ be a unital $C^*$-algebra, $L$ be a partially-defined $L$-seminorm on $\calA$. Then 
	$$
	\sup_\alpha\left\{\Diam(S(\calA_\alpha))\right\}=\Diam(S(\calA))=\Diam(QS(\calA))
	$$
	where $\alpha$ parametrizes all \textit{maximal} unital commutative $C^*$-subalgebras of $\calA$.
\end{proposition}
\begin{proof}
	Since the restriction map $(R_\alpha)_\#: S(\calA) \twoheadrightarrow S(\calA_\alpha)$, $(R_\alpha)_\#(\mu)=\mu_\alpha:=\mu|_{\calA_\alpha}$ is surjective, and the analogous map, defined on $QS(\calA)$,
	$$
	(R_\alpha)_\#: QS(\calA) \twoheadrightarrow S(\calA_\alpha),
	$$
	is surjective too, the equality $\sup_\alpha \left\{d_{L,\alpha}(\mu|_{\calA_\alpha},\nu|_{\calA_\alpha})\right\}=d_L(\mu,\nu)$, where $\alpha$ parametrizes all maximal unital commutative $C^*$-subalgebras, implies that
	\begin{multline*}
	\sup_\alpha\left\{\Diam(S(\calA_\alpha))\right\}=\sup_\alpha\left\{\sup\{d_{L,\alpha}(\mu_\alpha,\nu_\alpha): \mu_\alpha,\nu_\alpha\in S(\calA_\alpha)\}\right\}=\\
	 =\sup_\alpha\left\{\sup\{d_{L,\alpha}(\mu|_{\calA_\alpha},\nu|_{\calA_\alpha}): \mu,\nu\in S(\calA)\}\right\}=\\
	=\sup\{\sup_\alpha \left\{d_{L,\alpha}(\mu|_{\calA_\alpha},\nu|_{\calA_\alpha})\right\}: \mu,\nu\in S(\calA)\}=\\
	=\sup\{d_L(\mu,\nu): \mu,\nu\in S(\calA)\}=:\Diam(S(\calA))
	\end{multline*}
	By exactly the same argument,
	\begin{multline*}
	\sup_\alpha\left\{\Diam(S(\calA_\alpha))\right\}=\sup_\alpha\left\{\sup\{d_{L,\alpha}(\mu_\alpha,\nu_\alpha): \mu_\alpha,\nu_\alpha\in S(\calA_\alpha)\}\right\}=\\ 
	=\sup_\alpha\left\{\sup\{d_{L,\alpha}(\mu|_{\calA_\alpha},\nu|_{\calA_\alpha}): \mu,\nu\in QS(\calA)\right\}\}=\\
	=\sup\{\sup_\alpha \left\{d_{L,\alpha}(\mu|_{\calA_\alpha},\nu|_{\calA_\alpha})\right\}: \mu,\nu\in QS(\calA)\}=\\
	=\sup\{d_L(\mu,\nu): \mu,\nu\in QS(\calA)\}=:\Diam(QS(\calA))
	\end{multline*}
\end{proof}

\begin{proposition}
	\label{upper bound for W_p_projective}
	Let $\calA$ be a unital $C^*$-algebra, $L$ be a partially-defined $L$-seminorm on $\calA$. Then for every $p\in [1,\infty)$ and every $\mu,\nu\in QS(\calA)$
	$$
	\arWp{p}^p(\mu,\nu)\leq \Diam(S(\calA_\alpha))^{p-1}\cdot \arWp{1}(\mu,\nu)
	$$
\end{proposition}
\begin{proof}
	Recall, that by definition,
	$$
	W_{p,\alpha}^p(\mu_\alpha,\nu_\alpha):=\inf\left\{\int d_{L, \alpha}^p d\pi: \pi\in \Pi(\mu_\alpha,\nu_\alpha)\right\}
	$$
	For any $\pi\in \calP(X_\alpha\times X_\alpha)$ it is true, that 
	\begin{multline*}
	\int_{X_\alpha\times X_\alpha} d_{L, \alpha}^p(x,y) d\pi\leq \sup\{d_{L,\alpha}^{p-1}(x,y): x,y\in X_\alpha\}\cdot\int_{X_\alpha\times X_\alpha} d_{L, \alpha}(x,y) d\pi\leq\\
	\leq \sup\{d_{L,\alpha}^{p-1}(\tilde{\mu}_\alpha,\tilde{\nu}_\alpha): \tilde{\mu}_\alpha,\tilde{\nu}_\alpha\in S(\calA_\alpha)\} \cdot\int_{X_\alpha\times X_\alpha} d_{L, \alpha}(x,y) d\pi =\\
	=\Diam(S(\calA_\alpha))^{p-1}\cdot \int_{X_\alpha\times X_\alpha} d_{L, \alpha}(x,y) d\pi
	\end{multline*}
	Passing to the infimum over all $\pi\in \Pi(\mu_\alpha,\nu_\alpha)$, we obtain:
	$$
	W_{p,\alpha}^p(\mu_\alpha,\nu_\alpha)\leq \Diam(S(\calA_\alpha))^{p-1}\cdot W_{1,\alpha}(\mu_\alpha,\nu_\alpha)
	$$
	As follows from Proposition \ref{Diam>=Diam_alpha}, $W_{p,\alpha}^p(\mu_\alpha,\nu_\alpha)\leq \Diam(S(\calA))^{p-1}\cdot W_{1,\alpha}(\mu_\alpha,\nu_\alpha)$. Let us fix some $\mu,\nu\in QS(\calA)$ and take the supremum over all maximal unital commutative subalgebras $\calA_\alpha$. We obtain
	$$
	\arWp{p}^p(\mu,\nu)\leq \Diam(S(\calA))^{p-1}\cdot \arWp{1}(\mu,\nu)
	$$
\end{proof}

\begin{theorem}
	\label{W_proj metricizes weak-star}
	Let $\calA$ be a unital separable $C^*$-algebra, $L$ be a Wasserstein-compatible $\Lip$-seminorm on $\calA$. Then $\arWp{p}$ metricizes the weak$^*$-topology on $S(\calA)$ for every $p\in [1,\infty)$.
\end{theorem}
\begin{proof}
	Since $d_L$ metricizes the weak$^*$-topology on $S(\calA)$, and $S(\calA)$ is compact, $\Diam(S(\calA))$ is finite. By Proposition \ref{d_L=W_1_projective}, $\arWp{1}(\mu,\nu)=d_L(\mu,\nu)$ on $QS(\calA)$, hence $\arWp{1}(\mu,\nu)$ metricizes the Weak$^*$-topology on $S(\calA)$.
	By Proposition \ref{upper bound for W_p_projective}, 
	$$
	\arWp{p}\leq\Diam(S(\calA_\alpha))^{\frac{p-1}{p}}\cdot \arWp{1}^{\frac{1}{p}}
	$$
	By Proposition \ref{W_q>=W_p}, $W_{1,\alpha}\leq W_{p,\alpha}$ for $p\in [1,+\infty]$, $\calA_\alpha\in \calC(\calA)$. It follows, that
	$$
	\arWp{1}\leq\arWp{p}\leq\Diam(S(\calA_\alpha))^{\frac{p-1}{p}}\cdot \arWp{1}^{\frac{1}{p}}
	$$
	Thus, $\arWp{p}$ metricizes the same topology as $\arWp{1}$ does.
\end{proof}

Let us look on the projective $L^p$-Wasserstein distance from the categorical point of view. It is possible to define a category of all pairs $(K, d)$, where $K$ is a Hausdorff compact convex set and $d$ is a $[0,\infty]$-valued pseudo-distance function on $K$. Morphisms between $(K_1, d_1)$ and $(K_2, d_2)$ in this category are defined as continuous affine maps $f$ between $K_1$ and $K_2$ such that $d_2\circ f \leq d_1$ (in other words, $f$ is a contraction). Denote this category as $\bf CCHpd_{con}$. As $\bf CCH$ does, it contains all small projective limits.

\begin{proposition}
	For any small diagram $(\{(K_\alpha, d_\alpha)\}, \calT)$, $\calT=\{T_{\alpha,\beta}: K_\alpha \rightarrow K_\beta, \alpha\leq \beta\}$ in $\bf CCHpd_{con}$ there is a projective limit. The category $\bf CCHpd_{con}$ is complete.
	Projective limit in $\bf CCHpd_{con}$ is defined as a projective limit in $\bf CCH$, equipped with the following pseudo-distance function: $d:=\sup_\alpha \{d_\alpha \circ Pr_\alpha \}$. Projections $\pr_\alpha$ are the projection maps from the definition of the projective limit in $\bf CCH$. 
\end{proposition}
\begin{proof}
	It is obvious, that $d\geq d_\alpha\circ \pr_\alpha$, hence all $\pr_\alpha$ are morphisms in $\bf CCHpd_{con}$.
	
	Suppose we have an object $(K, d_K)$ and a family of morphisms $\varphi_\alpha: K\rightarrow K_\alpha$, such that $\varphi_\alpha=T_{\beta,\alpha}\circ \varphi_\beta$ iff $\alpha\leq \beta$ and $d_\alpha\circ \varphi_\alpha\leq d_K$. Define $\varphi: K\rightarrow \varprojlim K_\alpha$ as $\varphi(x)=(\varphi_\alpha(x))$. Then $(\pr_\alpha\circ \varphi)(x)=\pr_\alpha(\varphi_\alpha(x))=\varphi_\alpha(x)$, hence $\pr_\alpha\circ \varphi=\varphi_\alpha$ and 
	$$
	d_\alpha\circ\pr_\alpha\circ \varphi= d_\alpha\circ \varphi_\alpha\leq d_K
	$$
	Take the supremum over all $\alpha$ in both sides of the inequality:
	$$
	d_K\geq \sup_\alpha \{d_\alpha \circ \pr_\alpha \circ \varphi\} = d \circ \varphi
	$$
	Hence $\varphi$ is a morphism in $\bf CCHpd_{con}$. The uniqueness of this morphism follows from its uniqueness in $\bf CCH$ (see the proof of Proposition \ref{existence of inverse limits in CCH}).
\end{proof}

Let us define the category of all pairs $(\calA, L)$, where $\calA$ is a unital commutative $C^*$-algebra and $L$ is a partially-defined $L$-seminorm on $\calA$. Morphisms between $(\calA_1,L_1)$ and $(\calA_2, L_2)$ are injective unital $C^*$-homomorphisms between $\calA_1$ and $\calA_2$, such that $L_2(f(a))=L_1(a)$ $\forall a\in \calA$. Denote it by $\bf ucC^*L_{in}$.

Define a functor $(S,d)$ from $\bf ucC^*L_{in}^{op}$ to $\bf CCHpd_{con}$ that acts on objects as follows:
$$
(S,d)(\calA,L)=(S(\calA), d_L)
$$
and sends every morphism $f$ to its Banach adjoint $f^*$ restricted to $S(\calA)$.
Functoriality of this map follows from the fact that $d_L$ is a $[0,+\infty]$-valued pseudo-distance function, and that $d_{L,\alpha}\circ f^*\leq d_{L,\beta}$, if $f: \calA_\alpha \hookrightarrow \calA_\beta$ is a morphism in $\bf ucC^*L_{in}$. 

Analogously, we can define functors $(S,\Wass_p)$ from $\bf ucC^*L_{in}^{op}$ to $\bf CCHpd_{con}$ for every $p\in [1,\infty)$ that act on objects as follows:
$$
(S, \Wass_p)(\calA,L):=(S(\calA), W_p)
$$
and maps every morphism $f$ to its Banach adjoint $f^*$ restricted to $S(\calA)$. Here $W_p$ is the $L^p$-Wasserstein distance on $(S(\calA), W_p)$ w.r.t. pseudo-distance function $d_L$ on $\Spec(\calA)$.
The fact, that it is a functor, follows from Proposition \ref{W_p,alpha is distance}, which asserts that $W_p$ is a $[0,+\infty]$-valued pseudo-distance function, and Proposition \ref{W_beta>=W_alpha}, which implies that $W_p\circ f^*\leq W_p$ if $f: \calA_\alpha \hookrightarrow \calA_\beta$ is a morphism in $\bf ucC^*L_{in}$. 

Recall, that for a $C^*$-algebra $\calA$ we can consider a diagram $\calC(\calA)$ (in the category $\bf ucC^*_{in}$) of all unital commutative $C^*$-subalgebras of $\calA$ ordered by inclusion. Let us consider a partially-defined $\Lip$-seminorm $L$ on $\calA$ and equip each $\calA_\alpha\in \calC(\calA)$ with the restriction of $L$ to $\calA_\alpha$. We obtain a diagram $(\{(\calA_\alpha, L_\alpha)\}, \subseteq)$ in $\bf ucC^*L_{in}$.

Using the defined above functors, $(S,d)$ and $(S, \Wass_p)$, we can obtain diagrams in the category $\bf CCHpd_{con}$: $(\{(S(\calA_\alpha), d_{L,\alpha})\}, \calR_\#)$ and $(\{(S(\calA_\alpha), W_{p,\alpha})\}, \calR_\#)$. Projective limits of these diagrams are exactly the quasi-state spaces $(QS(\calA),d_L)$ and $(QS, \arWp{p})$. 

Resume the already known facts about these spaces:
\begin{itemize}
	\item If $L$ is a partially-defined $L$-seminorm on $\calA$, then $(QS(\calA),d_L)$ and $(QS, \arWp{p})$ for every $p\in [1,\infty)$ are compact convex Hausdorff spaces equipped with $[0,\infty]$-valued lower semi-continuous pseudo-distance functions. Proof: Prop. \ref{d_L is distance on QS} and \ref{W_proj is distance on QS}.
	\item If $L$ is an $L$-seminorm on $\calA$, then $(QS(\calA),d_L)$ is a compact convex space equipped with $[0,\infty]$-valued lower semi-continuous distance function. Proof: Prop. \ref{d_L is distance on QS}.
	\item If $L$ is a solid $L$-seminorm on $\calA$, then $(QS(\calA), \arWp{p})$ for every $p\in [1,\infty)$ is a compact convex space equipped with $[0,\infty]$-valued lower semi-continuous distance function. Proof: Prop. \ref{W_proj is distance on QS}.
	\item If $L$ is a Wasserstein-compatible $\Lip$-seminorm on $\calA$, then $(S(\calA), \arWp{p})$ for every $p\in [1,\infty)$ is a compact convex metric space, and $(QS(\calA),d_L)=(QS(\calA), \arWp{1})$. Proof: Prop. \ref{d_L=W_1_projective} and \ref{W_proj metricizes weak-star}.
\end{itemize}

\section{Some open problems}
\sectionmark{Some open problems}

We provided a rigorous definition for a space of quasi-states, $QS(\calA)$, where $\calA$ is a unital $C^*$-algebra. It appears to be a compact convex set. It is natural to ask, how to characterize compact convex sets that arise this way. Recall, that for the state spaces of $C^*$-algebras there is a complete characterization in the terms of their convex structure (see \cite{Alfsen}). It seems that there should be a connection (duality of some sort) between partial $C^*$-algebras, introduced in \cite{Heunen3}, and the desired notion of $QS$-spaces.

It is known, that the self-adjoint part of a $C^*$-algebra can be recovered from its state space (which is considered as an ordered Banach space). How much information is contained in a quasi-state space? Is it possible to recover a diagram of all unital commutative $C^*$-subalgebras knowing only the corresponding quasi-state space? 

There is a lack of non-trivial examples of solid $\Lip$-seminorms. Obviously, all pairs $(\calA,L)$, where $\calA$ is a unital $C^*$-algebra and $L$ is a finite $\Lip$-seminorm, are acceptable examples of solid $\Lip$-seminorms. The problem is to find a pair of a separable infinite-dimensional noncommutative $C^*$-algebra and a solid $\Lip$-seminorm defined on it, which is not everywhere finite. Recall that, by definition, $\Lip$-seminorm is solid iff it is finite on a \textquotedblleft weak dense" (in the sense of separation of points) subset of every maximal abelian $C^*$-subalgebra. This assumption seems natural, but it is difficult to verify. The explicit description of maximal abelian $C^*$-subalgebras is possible for some von Neumann algebras, but only finite-dimensional von Neumann algebras are separable as topological spaces. It is interesting to verify, whether the classical examples of quantum compact metric spaces, introduced by Connes and Rieffel, satisfy this property.

Speaking about Wasserstein-compatible $\Lip$-seminorms, there is a lack of examples of such seminorms even in finite-dimensional case. It is interesting to know, do the matrix algebras allow a $\Lip$-seminorm of this type.

\section*{Acknowledgement}
The author is extremely grateful to Dr. Frederic Latremoliere for pointing out a serious mistake in the first version of this paper.

The paper was written during an internship of the author in Scuola Normale Superiore di Pisa. The author is grateful to SNS for its hospitality, Dr. Luigi Ambrosio for interesting and fruitful mathematical conversations, and Higher School of Economics for a sponsorship of this visit.

The author wish to thank Alexander Kolesnikov for a constant support during the author’s research activity.

\end{document}